\documentclass[12pt,a4paper]{article} 
\usepackage[utf8]{inputenc} 
\usepackage[english]{babel} 
\usepackage{dsfont} 
 
\usepackage[ 
backend=bibtex, 
style=alphabetic, 
maxnames=10 
]{biblatex} 
\usepackage[colorlinks]{default} 
\usepackage{subcaption} 
 
\usepackage[left=2cm,right=2cm,top=2.5cm,bottom=2cm,headheight=15pt]{geometry} 
 
\title{Berry--Esseen-type estimates for random variables with a sparse dependency graph} 
\author{Maximilian \textsc{Janisch}\footnote{\scriptsize Institut für Mathematik, Universität Zürich. E-Mail: \url{maximilian.janisch@math.uzh.ch} or \url{mail@maximilianjanisch.com}.}\qquad Thomas \textsc{Leh{\'e}ricy}\footnote{\scriptsize Institut für Mathematik, Universität Zürich. E-Mail: \url{thomas.lehericy@math.uzh.ch}.}} 
\date{February 27, 2023} 
 
\usetikzlibrary{shapes.geometric,positioning} 
 
\DeclareMathOperator{\Cov}{Cov} 
\newtheorem{theorem}{Theorem} 
\newtheorem{corollary}[theorem]{Corollary} 
\newtheorem{conjecture}[theorem]{Conjecture} 
\newtheorem{lemma}[theorem]{Lemma} 
\newtheorem{proposition}[theorem]{Proposition} 
 
\theoremstyle{remark} 
\newtheorem{remark}[theorem]{Remark} 
 
\theoremstyle{definition} 
\newtheorem{definition}[theorem]{Definition} 
\newtheorem{example}[theorem]{Example} 
\newcommand{\eqdef}{\define} 
\newcommand{\ind}[1]{\mathds{1}_{#1}} 
\newcommand{\dkol}{\ensuremath{d_{\text{Kol}}}}

\renewcommand{\d}{\,\mathrm{d} } 
 
\newcommand{\cste}{\mathrm{const}\xspace} 
\newcommand{\e}{\mathrm{e}} 
 
\renewcommand{\i}{\ensuremath{\mathrm{i}}} 
\DeclareMathOperator{\argmin}{argmin} 
\newcommand{\ulim}[2]{\underset{#1\to #2}\longrightarrow}

\renewcommand{\abs}[1]{\left| #1 \right|}

\numberwithin{theorem}{section} 
\numberwithin{equation}{section} 
 
\newtheorem*{examples*}{Examples} 
 
\makeindex 
 
\begin{document} 
\maketitle

\begin{abstract} 
	We obtain Berry--Esseen-type bounds for the sum of random variables with a dependency graph and uniformly bounded moments of order $\delta \in (2,\infty]$ using a Fourier transform approach. Our bounds improve the state-of-the-art in the regime where the degree of the dependency graph is large. As a Corollary of our results, we obtain a Central Limit Theorem for random variables with a sparse dependency graph that are uniformly bounded in $L^{\delta}$ for some $\delta\in(2,\infty]$. 
\end{abstract}

\tableofcontents

\section{Introduction}

\index{aal-S@$S$} 
\index{aal-N@$N$} 
\index{aal-v@$v\define\sqrt{\V[S]}$} 
\index{aal-Yk@$(Y_k)_{k\in\N}$} 
\index{aal-rv@r.v.: random variable} 
\index{aal-iid@i.i.d.: independent and identically distributed} 
 
Consider 
$(Y_{k})_{k \in V}$ a family of $L^2$ real random variables (r.v.) indexed by a finite set $V$ of cardinality $N\in\N=\Z_{\ge 1}$, and let $S \eqdef \sum_{k \in V} Y_k$. We do not assume that they are independent. 
We want to study if and at what speed the random variable 
\begin{equation}\label{eq:W} 
W \define \frac{S-\E[S]}{\sqrt{\V[S]}} 
\end{equation} 
converges to a normal distribution. Our results are non-asymptotic, meaning that $V$ is fixed; our results can therefore be directly adapted to hold for triangular arrays $(Y_{m,k})_{m\in\N,k\in\set{1,\dots,N(m)}}$ of random variables for a fixed sequence $(N(m))_{m\in\N}$ in $\N$ going to $\infty$. 
 
In the simplest situation, $V = \{1, \dots, N\}$ and the $(Y_k)_{k\in V}$ are independent and identically distributed (i.i.d.). The central limit theorem ensures that $W$ 
converges, in distribution as $N\to\infty$, to a standard-normally distributed random variable. If the $(Y_k)_{k\in V}$ have a finite third moment we can bound the speed at which the distribution of $W$ converges to the standard normal distribution $\kN(0,1)$. The usual metric used to quantify this speed of convergence is the \emph{Kolmogorov distance}, defined for two real random variables $X,Y$ with distributions $\P_X,\P_Y$ as \index{aal-dkol@$\dkol(X,Y)\define\sup_{t\in\R}\abs{\P[X\le t]-\P[Y\le t]}$} 
\begin{equation*} 
	\dkol(X,Y)=\dkol(X, \P_Y)=\dkol(\P_X, Y)=\dkol(\P_X,\P_Y)\define\sup_{t\in\R}\abs{\P(X\le t)-\P(Y\le t)}. 
\end{equation*} 
Note that $\dkol(X,Y)\leq 1$ for every $(X,Y)$. 
Other metrics can be used, for example the Wasserstein-1 distance. 
Berry and Esseen proved independently that 
\begin{equation}\label{eq:Berry-Esseen-classical} 
	\dkol(W, \kN(0,1))\le \frac{C}{\sqrt{N}}\frac{\E[\abs{Y_1}^3]}{\V[Y_1]^{\frac 32}}. 
\end{equation} 
Here $C$ is a constant independent of $V$ and $(Y_k)_{k\in V}$. Esseen showed that \eqref{eq:Berry-Esseen-classical} holds for $C=7.6$. The constant has been improved to $C\leq 0.5583$, cf. \cite{Shevtsova2013}. 
The case of independent but non-identically distributed r.v. is well-known: the central limit theorem was generalized by Lindeberg in 1922 under his eponymous condition, and 
\cite{Shevtsova2013} gives for $C=0.5583$ 
\begin{equation}\label{eq:Berry-Esseen-non-iid} 
	\dkol(W,\kN(0,1))\le C \frac{\sum_{k\in V}\E[\abs{Y_k}^3]}{(\V[S])^\frac32}. 
\end{equation} 
For a different constant, \eqref{eq:Berry-Esseen-non-iid} was proven in \cite{Esseen1944}.

The case where $(Y_k)_{k\in V}$ are not independent is much more difficult. In general, we do not even have convergence (as $N\to\infty$) of $W$ to a normal distribution if we replace the condition of independence of the $(Y_k)_{k\in V}$ by pairwise independence. Indeed, as shown in \cite{Janson1988}, there exists a stationary sequence $(Y_n)_{n\in\N}$ of pairwise independent real $L^\infty$-random variables with strictly positive variance such that $W$ defined above converges in distribution to $0$ as $n\to\infty$, not to a standard normal distribution. In fact, the situation is even worse: For any fixed $m\in\N$, there exists a sequence of identically distributed real $L^\infty$-random variables $(Y_n)_{n\in\N}$ such that any sub-family $(Y_{n_1},Y_{n_2},\dots, Y_{n_m})$ of $(Y_n)_{n\in\N}$ with $\set{n_1,\dots, n_m}\subset\N$ is independent, but such that $(W_n)_{n\in\N}$ does not satisfy the Central Limit Theorem. See \cite{Pruss1998} for a proof of this result. 
 
The need to specify the \emph{way} the random variables depend on each other is therefore crucial. We are interested in dependency graphs \cite{Rinott89, feray2017mod}, a slight generalization of which 
is the notion of local dependence \cite{ChenShao}; 
other notions include weighted dependency graphs \cite{feray2018weighted}, Markov chains and fields, 
or functions of independent random variables \cite{daniel2014phd}, among others. Our focus lies in Berry--Esseen-type bounds, which naturally yield a Central Limit Theorem, albeit with suboptimal conditions; for Central Limit Theorems for random variables which are only partially independent, see for example \cite{Orey1958} and more recently \cite{Janson2021}.

\paragraph{Dependency graph.} 
A simple graph $G$ on $V$ (its \emph{vertex set}) is an ordered tuple $(V,E)$, where $E$ is a set of pairs (called \emph{edges}) of distinct vertices of $V$. If two vertices belong to the same edge, we say that they are \emph{adjacent} to each other. If a vertex $v$ belongs to an edge $e$, we say that \emph{$v$~is~incident~to~$e$}. The \emph{degree} of a vertex is the number of edges it belongs to. 
We also consider multigraphs, where multiple edges are allowed: this corresponds to allowing $E$ to be a multiset. In the case of multigraphs, we also allow loops, i.e. edges whose endpoints are the same vertex. 
 
We say that a multigraph $G$ is a \emph{dependency graph of $(Y_k)_{k \in V}$} if and only if the following property holds: If $V_1,V_2\subset V$ are disjoint and there exist no edge with one vertex in $V_1$ and the other in $V_2$, then the families of random variables $(Y_k)_{k\in V_1}$ and $(Y_k)_{k\in V_2}$ are independent. For a deeper introduction to dependency graphs, see for example the remarks following \cite[Definition 9.1.1]{feray2013mod}. 
 
Dependency graphs can describe independent random variables: $(Y_{k})_{k \in V}$ is independent if and only if they admit the graph on $V$ with no edge as a dependency graph.

Our main results concerns $\dkol(W,\kN(0,1))$ when the $(Y_k)_{k\in V}$ admit a dependency graph. The bound is state-of-the-art when the variables are in $L^\infty$, and beats the state-of-the-art in certain regimes when the variables are in $L^\delta$ for some $\delta<\infty$, see Section \ref{sec:direct-comparison}. 
For conciseness, we write $a\lesssim b$ to mean $a\le K b$ for an absolute constant $K\ge 0$. 
 
Throughout the text, we assume that the family $(Y_k)_{k\in V}$ has a dependency graph of maximal degree $D\in\Z_{\ge 0}$. 
\begin{definition}\label{def:renormedSD} 
	Assume that for some $\delta\in(2,\infty)$, all $Y_k$ are in $L^\delta$. Define the \emph{renormalized standard deviation} 
	\begin{equation}\label{eq:xi-delta} 
		\xi_\delta\define\left(\frac{N}{\kA_\delta}\right)^{\frac 1\delta}\sqrt{\frac{\V[S]}{N(D+1)}}, 
	\end{equation} 
	where 
	\begin{equation*} 
		\kA_\delta =\sum_{k\in V}\E[\abs{Y_k-c_k}^\delta] 
	\end{equation*} 
	\index{aal-ck@$c_k$} 
	\index{aal-D@$D$} 
	\index{aal-calAk@$\kA_\delta\define\sum_{k\in V}\E[\abs{Y_k-c_k}^\delta]$} 
	\index{aag-xidelta@$\xi_\delta\define\left(\frac N{\kA_\delta}\right)^{\frac 1\delta}\sqrt{\frac{\V[S]}{N(D+1)}}$} 
	for an arbitrary family of real numbers $(c_k)_{k\in V}$. 
\end{definition} 
 
Proposition \ref{prop:decrease-xi} guarantees that $\xi_\delta \in [0,1]$ and that $\xi_\delta$ is a decreasing function of $\delta$. 
In what follows we fix the sequence $(c_k)_{k\in V}$. 
Reasonable choices are $c_k = 0$ or $c_k = \E[Y_k]$, but other values could be considered, for example $c_k \in 
\argmin_c \| Y_k - c \|_\delta$. 
 
The Berry--Esseen inequality \eqref{eq:Berry-Esseen-non-iid} for independent random variables with finite moments of order $L^\delta$ for some $\delta\in(2,3]$ is, up to the absolute constant, equivalent to the following Theorem. 
\begin{theorem}[Cf. {\cite[Theorem V.3.6]{Petrov}}]\label{thm:Berry-Essen-general} 
	Let $(Y_k)_{k\in V}$ be a family of $\abs{V}=N$ independent random variables (in particular they admit the empty dependency graph with $D=0$). Assume that $\kA_\delta<\infty$ for some $\delta\in(2,3]$. Then 
	\begin{equation}\label{eq:Berry-Esseen-general} 
	\dkol\left(W, \kN(0,1)\right)\lesssim \xi_\delta^{-\delta} \left(\frac 1N\right)^{\frac{\delta-2}2}. 
	\end{equation} 
\end{theorem}

We conjecture that \eqref{eq:Berry-Esseen-general} remains true even for random variables with a dependency graph, if the term $\frac 1N$ is replaced with $\frac{D+1}N$. 
 
\begin{conjecture}\label{conj:main} 
	Let $(Y_k)_{k\in V}$ be a family of $\abs{V}=N$ random variables admitting a dependency graph of degree $D\in\Z_{\ge 0}$. Assume that there exists $\delta\in(2,3]$ such that $\kA_\delta<\infty$. Then 
	\begin{equation}\label{eq:main-conj} 
	\dkol\left(W, \kN(0,1)\right)\lesssim \xi_\delta^{-\delta} \left(\frac {D+1}N\right)^{\frac{\delta-2}2}. 
	\end{equation} 
\end{conjecture} 
 
While the conjecture remains out of reach of current techniques, both ours and Stein’s method, our results provide evidence in favor of the conjecture. In particular, our results support the heuristic that the substitution of $1/N$ by $(D+1)/N$ is the correct adaptation. We briefly present arguments in favor of the conjecture at the end of Section \ref{sec:optimality}. 
 
We present our results in Section \ref{sec:main-results}, and investigate the regime where they do better than the current state-of-the-art (mostly obtained by Stein’s method) in Section \ref{sec:previous-results}. The optimality of our results is discussed in Section \ref{sec:optimality}. Section \ref{sec:applications} gathers some applications of our results: Central Limit Theorems in Section \ref{sec:CLTs}, U-statistics in Section \ref{sec:U-statistics}, and volatility estimation for prices of financial assets in Section \ref{sec:stock-volatility}. We gather the proof of our main results in Section \ref{sec:refined}. A concise explanation of our proof ideas can be found in Section \ref{sec:overview}.

\section{Main results} 
\label{sec:main-results} 
We assume that the random variables all belong to some $L^\delta$ for $\delta\in(2,\infty]$, equipped with $\| \cdot \|_\delta$, the $L^\delta$ norm. Define $v = \sqrt{\V[S]}$. Recall $\kA_\delta$ and $\xi_\delta$ as in Definition \ref{def:renormedSD}. In Theorem \ref{thm:mod-phi-corollary-new}, which has the sharpest bound, we consider the case $\delta=\infty$; in Theorem \ref{thm:dkol-delta-3+}, the case $\delta \in [3,\infty)$; and in Theorem \ref{thm:result-refined-delta<2}, the case $\delta\in(2,3]$.

\begin{theorem}\label{thm:mod-phi-corollary-new} 
	Assume that $v>0$ and that 
	 $\max_{k\in V} \|Y_k-c_k\|_\infty \leq L$ for some $L\in(0,\infty)$. 
	Then\index{aag-deltaprime@$\delta'\define\min(\delta,3)$} 
	\begin{align}\label{eq:mod-phi-corollary-new} 
		\dkol\left(W, \kN(0,1)\right) \leq \max\left\{ 68.5 \frac{(D+1)^2 \kA_3}{v^3} \ , \ 22.88 \frac{L(D+1)}{v} \right\} . 
	\end{align} 
\end{theorem} 
 
This result is essentially optimal, as discussed in section \ref{sec:optimality}. 
 
\begin{remark}\label{rem:sigma} 
	Theorem \ref{thm:mod-phi-corollary-new} gives a strictly better conclusion (except for the constant) than \cite[Corollary 30]{feray2017mod}, which gives the bound 
	\begin{equation}\label{eq:FMN-estimate} 
	\dkol \left(W, \kN(0,1)\right) \le 76.36 \frac{L^3 N (D+1)^2}{v^{3}}. 
	\end{equation} 
	Indeed, in this case, define 
	 
	$$M_\delta = \max_{k\in V}\,\norm{Y_k-c_k}_\delta \le L \qquad , \qquad \sigma_\delta=\frac 1{M_\delta} \sqrt{\frac{v^2}{N(D+1)}} \in [0,1]. $$ 
	Then 
	\begin{align} 
	\nonumber \max\left\{ \frac{(D+1)^2 \kA_3}{v^3}, \frac{L(D+1)}{v} \right\} &= \sqrt{\frac{D+1}N}\max\left\{\frac{1}{\xi_3^3}\ ,\  \frac{1}{\xi_\delta} \frac{L}{\left(\frac{\kA_\delta}N\right)^\frac{1}\delta}\right\} \\ 
	&\le \sqrt{\frac{D+1}N}\max\left\{\frac 1{\sigma_3^3}, \frac 1{\sigma_3}\frac L{M_3} \right\} \label{eq:infty-sigma} 
	\end{align} 
	since $\xi_\delta (\kA_\delta / N)^{1/\delta} = \sigma_3 M_3$ and $\sigma_3 \leq \xi_3 \leq 1$, and 
	\begin{align*} 
	\frac{L^3 N (D+1)^2}{v^{3}} &= \sqrt{\frac{D+1}{N}}\frac{1}{\sigma_3^3}\left(\frac L{M_3}\right)^3 \geq \sqrt{\frac{D+1}N}\max\left\{\frac 1{\sigma_3^3}, \frac 1{\sigma_3}\frac L{M_3} \right\} 
	\end{align*} 
	since $L\geq M_3$. 
\end{remark} 
 
\begin{remark}\label{rem:Rinott} 
	Under the same assumptions of Theorem \ref{thm:mod-phi-corollary-new} with $c_k=\E[Y_k]$, \cite{Rinott} obtains 
	\begin{equation}\label{eq:Rinott-bound} 
		\dkol\left(W,\kN(0,1)\right)\lesssim\frac 1v\left((D+1)L+\sqrt{\frac{N}{v^2}} (D+1)^\frac32 L^2 + \frac{n}{v^2}(D+1)^2 L^3\right). 
	\end{equation} 
	This is a Corollary of Theorem \ref{thm:mod-phi-corollary-new} because $\kA_3\le N L^3$. Therefore the right-hand side of \eqref{eq:mod-phi-corollary-new} is less or equal than an absolute constant times the maximum of the first term of \eqref{eq:Rinott-bound} and the third term of \eqref{eq:Rinott-bound}. 
\end{remark} 
 
At the cost of a more complicated inequality, Theorem \ref{thm:mod-phi-corollary-new} can be improved to have a more precise upper bound that includes the third centered moment of $S$ (or more generally moments up to any fixed order). This is done in the following Theorem. 
\begin{theorem}\label{thm:mod-phi-corollary-new-exact} 
		Assume that $v>0$ and $\max_{k\in V} \|Y_k-c_k\|_\infty \leq L$ as in Theorem \ref{thm:mod-phi-corollary-new}. Let $\rho = |\E[(S-\E[S])^3]|$. 
		Then 
		\begin{multline}\label{eq:thm-new-exact} 
\dkol\left(W, \kN(0,1)\right) \leq 
 0.607148 \frac{\rho}{v^3} + 116.84 \frac{(D+1)^3 \kA_4}{v^4} \\ 
+ \max \left\{ 16.57\frac{L(D+1)}{v} \ , \ 22.47 \sqrt{\frac{(D+1)^3 \kA_4}{v^4}} + 1.596 \frac{\rho}{v^3} \right\} \ . 
		\end{multline} 
		If the bound of Theorem \ref{thm:mod-phi-corollary-new} is smaller than $1$, then the right-hand side of \eqref{eq:thm-new-exact} is smaller than that of Theorem \ref{thm:mod-phi-corollary-new}, multiplied by $1.06164$. When $c_k = \E[Y_k]$ for every $k\in V$, this multiplicative constant is instead $0.85771$. 
\end{theorem} 
 
We expect the bound of Theorem \ref{thm:mod-phi-corollary-new-exact} to always be better than that of Theorem \ref{thm:mod-phi-corollary-new}, but were only able to establish this up to a multiplicative constant of $1.06164$.

We finally treat the case where the random variables only need to have a finite third moment (meaning $\kA_3<\infty$), or even just a finite $2+\varepsilon$-th moment for some $\varepsilon>0$ (meaning $\kA_\delta<\infty$ for some $\delta>2$). 
\begin{theorem}\label{thm:dkol-delta-3+} 
	Assume $v>0$ and $\kA_\delta<\infty$ for some $\delta\in[3,\infty)$. Then 
	\begin{equation}\label{eq:dkol-3+-sigma-large} 
	\dkol\left(W,\kN(0,1)\right)  \leq 
	\max\left\{ 18.96 \, \xi_\delta^{-\frac{\delta}{\delta+1}}\left(\frac{D+1}N\right)^{\frac{\delta-2}{2(\delta+1)}}  \ , \  \frac{227.5}{\xi_\delta^3} \sqrt{\frac{D+1}{N}} \right\} \ . 
	\end{equation} 
\end{theorem}

\begin{theorem}\label{thm:result-refined-delta<2} 
	Assume $v>0$ and $\kA_\delta<\infty$ for some $\delta \in (2,3)$. 
	Then 
	\begin{equation}\label{eq:result-refined-delta<2} 
	\dkol\left(W,\kN(0,1)\right) \leq 8.015 \ \xi_\delta^{-\frac{\delta}{\delta+1}}\left(\frac{D+1}N\right)^{\frac{\delta-2}{2(\delta+1)}}  . 
	\end{equation} 
\end{theorem} 
 
\begin{remark} 
We conjecture that the bounds of Theorems \ref{thm:dkol-delta-3+} and \ref{thm:result-refined-delta<2} are not optimal, see Conjecture \ref{conj:main}. Consequently, we have done little to optimize the constants in the theorems above. We are confident that many of the techniques used to improve the bounds of the i.i.d. Berry--Esseen theorem (see e.g. \cite{Shevtsova2013}) can be adapted to our setting. 
\end{remark} 
 
\medskip 
 
\begin{remark} 
The asymptotic behaviour of the bounds of Theorems \ref{thm:dkol-delta-3+} and \ref{thm:result-refined-delta<2} as $N\to\infty$ is not always obvious, since $\xi_\delta$ implicitely depends on $N$. For the same reason, it is not always clear which $\delta$ gives the best bound. 
 
As noted after Definition \ref{def:renormedSD}, $\xi_\delta \in [0,1]$ and $\xi_\delta$ decreases with $\delta$. This means that the terms $\xi_\delta^{-\delta/(\delta+1)}$ and $\xi_\delta^{-3}$ become larger as $\delta\to\infty$. Compare this to Theorems \ref{thm:Stein-result} and \ref{thm:Stein-dkol}, where only $\xi_3$ and $\xi_4$ are needed. 
On the other hand, a larger $\delta$ gives a better exponent on $((D+1)/N)^{\frac{\delta-2}{2(\delta+1)}}$. 
 
If the $(Y_k)_{k\in V}$ are identically distributed and $\frac{\V[S]}{N(D+1)}$ is fixed, then $\xi_\delta$ does not depend on $N$, and only the term $(D+1)/N$ matters. In this case, fixing $\delta$ as large as possible gives the best asymptotic bounds as $N/(D+1)\to\infty$ --- it is useful to note that the maximum in \eqref{eq:dkol-3+-sigma-large} equals the first term when $N/(D+1)$ is large enough. 
 
The situation is different when the $(Y_k)_{k\in V}$ are not identically distributed. Let us highlight that the definition of $\xi_\delta$ involves 
$$\frac{\mathcal{A}_\delta}N = \frac{1}{N}\sum_{k\in V} \E[|Y_k - c_k|^\delta] ,$$ 
the average $L^\delta$ norm of the $(Y_k-c_k)_{k\in V}$. This is significantly better than a more naive bound using $\sup_{k\in V} \E[|Y_k-c_k|^\delta]$ instead of $\mathcal{A}_\delta/N$, particularly in the case where the $(Y_k)_{k\in V}$ are not identically distributed. 
 
Let us give an example where taking $\delta$ as large as possible is not the best decision. Consider independent $(Y_k)_{k\in V}$ with $Y_k \sim \mathrm{Ber}(1/k)$ and $c_k = \E[Y_k] = 1/k$ for every $k\geq 1$. Then $\sup_{k\in V} \E[|Y_k-c_k|^\delta] \geq \E[|Y_2-1/2|^\delta] = 2^{-\delta}$, but 
$$ \frac{\mathcal{A}_\delta}{N} = \frac{\ln N}{N} (1+o_N(1)) = o_N(1) \quad , \quad \V[S] = \ln N (1+o_N(1)) \quad , \quad \xi_\delta = \left(\frac{\ln N}{N}\right)^{\frac{1}{2}-\frac{1}{\delta}} (1+o_N(1)). $$ 
In this case, Theorem \ref{thm:dkol-delta-3+} gives 
\begin{equation*} 
\dkol(W, \mathcal{N}(0,1)) \lesssim \max\left\{ (\ln N)^{-\frac{\delta-2}{2(\delta+1)}} \ , \ \frac{N^{1-\frac{3}{\delta}}}{(\ln N)^{\frac{3(\delta-2)}{2\delta}}} \right\} . 
\end{equation*} 
The bound is non-trivial only when $\delta=3$, in which case $\dkol(W,\mathcal{N}(0,1)) \lesssim (\ln N)^{-1/8}$. On the other hand, the bound of Theorem \ref{thm:result-refined-delta<2} becomes 
\begin{equation*} 
\dkol(W,\mathcal{N}(0,1)) \lesssim (\ln N)^{-\frac{\delta-2}{2(\delta+1)}} . 
\end{equation*} 
The better bound is attained for $\delta=3$, and is then identical to that of Theorem \ref{thm:dkol-delta-3+}.

\end{remark} 
 
\begin{remark} 
The best bound in the example in the previous remark is given by Theorem \ref{thm:mod-phi-corollary-new}: 
$$ \dkol(W,\mathcal{N}(0,1)) \lesssim (\ln N)^{-1/2} . $$ 
In fact, for most applications, Theorem \ref{thm:mod-phi-corollary-new} gives better bounds than Theorems \ref{thm:dkol-delta-3+} and \ref{thm:result-refined-delta<2} if it can be applied. 
\end{remark}

\section{Comparison to the state-of-the-art} 
\label{sec:previous-results} 
\label{sec:direct-comparison}

Our approach, which is close to that of the first proof of Berry--Esseen, and of the state-of-the art in the i.i.d. case, centers around bounding the Fourier transform (or characteristic function) $\E[\e^{\i \xi W}]$ of $W$. Recall that $\E[\e^{\i \xi X_n}]$ converges pointwise (i.e. for every $\xi\in\R$) towards $\e^{-\xi^2/2}$ as $n\to\infty$ if and only if $X_n$ converges in distribution towards $\kN(0,1)$. Hence, it does not come as a surprise that any uniform bound on the difference between the Fourier transform of $W$ and that of $\kN(0,1)$ on a neighborhood of $0$ translates to a bound on $\dkol(W,\kN(0,1))$, see Section \ref{sec:overview}. We can in fact derive much more than a Berry--Esseen-type bound from a good enough control on the Fourier transform of $W$, which is the main appeal of this method. Its drawback is that it is hard to control the Fourier transform of $W$ when the $(Y_k)_{k\in V}$ are not independent.

An alternative approach was pioneered by Stein \cite{Stein72} and met resounding success thanks to its polyvalence. Stein’s method works with distances of the form 
$$ \mathrm{d}_\kF(W , G) = \sup_{f\in \kF} | \E[f(W)] - \E[f(G)] | $$ 
for $\kF$ some class of functions that is large enough for the above expression to indeed define a distance: for example, the set of all $1$-Lipschitz functions defines the Wasserstein-1 distance $W_1$, and the set of all $\ind{(-\infty,t]}$ for $t\in\R$ defines the Kolmogorov distance. The right-hand side is then re-written by introducing well-chosen couplings, then bounded, typically by Taylor expansion and using a Lipschitz property for $f$ and its derivative. As an example, \cite[Theorem 3.5]{ross2011fundamentals} or \cite{Rinott89} yields, for $(Y_k)_{k\in V}$ having a dependency graph of maximum degree $D$, 
\begin{equation}\label{eq:Stein-W1} 
W_1(W , \kN(0,1)) \leq \frac{D^2}{\V[S]^{3/2}} \sum_{k\in V} \E[|Y_k|^3] + \sqrt{\frac{26}{\pi}} \frac{D^{3/2}}{\V[S]} \sqrt{\sum_{k\in V} \E[Y_k^4]} . 
\end{equation} 
Bounds in the Kolmogorov distance are trickier to obtain, crucially because the functions $\ind{(-\infty, t]}$ have a discontinuity. An easy, but suboptimal bound is derived from \eqref{eq:Stein-W1} together with $\dkol(W, \kN(0,1)) \leq 2 \sqrt{C W_1(W, \kN(0,1))}$ where $C = 1/\sqrt{2\pi}$ is the maximum of the density of $\kN(0,1)$. Such a bound is found in \cite{Penrose2003}: 
\begin{theorem}[{\cite[Theorem 2.4]{Penrose2003}}]\label{thm:Stein-result} 
	Let $(Y_k)_{k\in V}$ be a finite family of real $L^4$-random variables admitting a dependency graph with maximal degree $D\in\N$. Then, if $\V[S]>0$, 
	\begin{equation*} 
		\dkol\left(W, \kN(0,1)\right)\leq 6 \frac{(D+1)}{\V[S]^\frac 34}\sqrt{\sum_{k\in V}\E[\abs{Y_k-\E[Y_k]}^3]}+ 6 \frac{(D+1)^\frac 32}{\V[S]}\sqrt{\sum_{k\in V}\E[\abs{Y_k-\E[Y_k]}^4]}. 
	\end{equation*} 
\end{theorem} 
 
Surprisingly, it is better in some cases than \cite[Theorem 2.7]{ChenShao}, which is proven using the slightly more general framework of local dependence: 
\begin{theorem}\label{thm:Stein-dkol} 
Fix $\delta \in (2,3]$. 
Let $V$ be a set of cardinality $N$ and $(Y_k)_{k\in V}$ be a family of random variables that admit a dependency graph of maximum degree $D$, such that $\E[Y_k]=0$ and $\|Y_k\|_\delta \leq \theta < \infty$ for every $k\in V$. Then 
$$ \dkol(W,\kN(0,1)) \leq 75 N (D+1)^{5(\delta-1)} \theta^\delta = 75 \sigma_\delta^{-\delta} \left(\frac {D+1}N\right)^{\frac{\delta-2}2} (D+1)^{4\delta-4} , $$ 
where $\sigma_\delta$ is as in Remark \ref{rem:sigma} with $c_k = 0$ for every $k\in V$. 
\end{theorem}

The dependence in $N$ is optimal (i.e. just as good as in the classical Berry--Esseen inequality, Theorem \ref{thm:Berry-Essen-general}). However, the exponent of $D$ is significantly worse, and there are regimes (namely $D$ large) where our Theorems \ref{thm:dkol-delta-3+} and \ref{thm:result-refined-delta<2} do significantly better than Theorem \ref{thm:Stein-dkol}. In addition, $\sigma_\delta$ is used instead of $\xi_\delta$, which is worse since we always have $\sigma_\delta\le\xi_\delta$. Nevertheless, for $D$ small, Theorem \ref{thm:Stein-dkol} constitutes, as far as we know, the state-of-the-art for Berry--Esseen-type bounds in the context of dependency graphs.

The former state-of-the-art using the Fourier transform method, using a similar starting point as our results, was as follows. We strictly improve this result. 
\begin{theorem}[{\cite[Theorem 39]{feray2017mod}}]\label{thm:modII-result} 
	Let $(Y_{n,k})_{1\leq k \leq n}$ be a triangular array of real random variables such that $A\define\sup_{1\leq k \leq n}\norm{Y_{n,k}}_{\delta}<\infty$ for some $\delta\in(6,\infty)$, where, for every $\delta \in [1,\infty]$, $\norm\cdot_{\delta}$ denotes the $L^\delta$ norm on random variables. Assume furthermore that $(Y_{n,k})_{1\le k\le n}$ admits a dependency graph with maximal degree $D=D(n)\in\N$ for every $n\in\N$, and assume that for $S_n\define\sum_{k=1}^n Y_{n,k}$, we have $V[S_n]>0$ as well as 
	\begin{equation}\label{eq:condition-FMN2} 
		\frac{n^{\frac{3+\delta}{3\delta}}(D+1)^\frac 23}{\sqrt{\V[S_n]}}\to 0 
	\end{equation} 
	for $n\to\infty$. Then for $n$ large enough, 
	\begin{equation} 
		\dkol\left(\frac{S_n-\E[S_n]}{\sqrt{\V[S_n]}}, \kN(0,1)\right)\lesssim \left(\frac{A n^{\frac{3+\delta}{3\delta}}(D+1)^\frac 23}{\sqrt{\V[S_n]}}\right)^{\frac{3\delta}{\delta+3}}. 
		\label{eq:FMN-bound} 
	\end{equation} 
\end{theorem}

\paragraph{Graphical comparison.} We illustrate the relative merits of the bounds in Theorems \ref{thm:mod-phi-corollary-new} to \ref{thm:modII-result} on an example. Since we believe (see Conjecture \ref{conj:main}) that the optimal rate has not been reached for the Kolmogorov distance for sums of random variables with a dependency graph, unlike in the i.i.d. case, we care little about constants, and we do not consider them. 
The difference to the i.i.d. case is most visible either when $\V[S]$ is large, or when it is small. We consider the case of large variance. 
 
Partition $V$ (of size $N = n(D+1)$) into $n$ sets of size $D+1$ such that $k\neq \ell\in V$ are adjacent if and only if they are in the same set of the partition. Fix $(Y_k)_{k\in V}$ with unit variance such that $\Cov(Y_k, Y_\ell)=1$ for every $k,\ell\in V$ such that $k$ is incident to $\ell$. Then $\V[S] = N(D+1)$. Fix $c_k = 0$ for every $k\in V$, and define $M_\delta = \sup_{k\in V}\norm{Y_k}_\delta = (\kA_\delta / N)^{1/\delta}$.

Theorem \ref{thm:Stein-result} gives 
\begin{equation*} 
\dkol(W,\kN(0,1)) \lesssim \left(\frac{D+1}{N}\right)^{\frac{1}{4}} M_\delta^{\frac{3}{2}} + \frac{D+1}{N^{\frac{1}{4}}} M_\delta^2 \lesssim \frac{D+1}{N^{1/4}} (M_\delta^2 + M_\delta^{3/2}) , 
\end{equation*} 
whenever $\delta\geq 4$. 
Theorem \ref{thm:Stein-dkol} gives 
\begin{equation*}\begin{split} 
	\dkol(W, \kN(0,1))&\lesssim N(D+1)^{5(\delta-1)} \frac{M_\delta^\delta}{\V[S]^\delta} \lesssim  N^{1-\frac{\delta}{2}} D^{\frac{9}{2}\delta-5} M_\delta^\delta . 
\end{split}\end{equation*} 
For $\delta> 6$, Theorem \ref{thm:modII-result} gives 
\begin{equation*} 
	\dkol(W, \kN(0,1))\lesssim M_\delta^{\frac{3\delta}{\delta+3}}N^{\frac{6-\delta}{2(\delta+3)}} (D+1)^{\frac{\delta}{2(\delta+3)}} 
\end{equation*} 
Theorems \ref{thm:dkol-delta-3+} and \ref{thm:result-refined-delta<2} give 
\begin{equation*} 
	\dkol(W, \kN(0,1))\lesssim \max\left\{M_\delta^{\frac\delta{\delta+1}}, M_\delta^3\right\} \left(\frac{D+1}{N}\right)^{\frac{\delta-2}{2(\delta+1)}} 
\end{equation*} 
for every $\delta>2$. 
 
Figure \ref{fig:convergencespeedplot} gives a comparison of the exponent of $N$ in the results discussed here. The figure is however useful only if $D$ is very small (i.e. typically less than a constant independent of $N$): while the bound of Theorem \ref{thm:Stein-dkol} is essentially optimal for $D$ small, it becomes significantly worse for $D$ large. 
 
For a comparison which also incorporates the size of $D+1$, we determine the best bounds depending on $\delta$ and on $\alpha \define \ln(D+1)/\ln N$. 
In Figure \ref{fig:comparison-D-N}, 
we apply Theorems \ref{thm:dkol-delta-3+} and \ref{thm:result-refined-delta<2} (green), Theorem \ref{thm:Stein-result} (orange) and Theorem \ref{thm:Stein-dkol} (blue) for the case $\alpha\in[0,0.1]$ ($y$-axis) and for $\delta\in(2,10]$ ($x$-axis). Each Theorem gives, for this setup, an upper bound on $\dkol(W, \kN(0,1))$ of the form $C N^x$ for some unique $x\in\mathbb R\cup\set{\infty}$, where $C$ is a constant depending only on the family of random variables and $\delta$, but not on $N$. In the plot, we show what Theorem gives the best (i.e. lowest) $x$ for what regions of $(\delta,\alpha)$. Note that for our Theorems \ref{thm:dkol-delta-3+} and \ref{thm:result-refined-delta<2}, the exponent $x$ equals $JL(\delta,\alpha)\define\frac{(\alpha-1)(\delta-2)}{2(\delta+1)}$, while for Penrose’s Theorem \ref{thm:Stein-result}, the exponent $x$ equals $P(\delta,\alpha)\define \alpha-\frac 14+\infty\mathds 1_{\delta<4}$ and for Chen and Shao’s Theorem \ref{thm:Stein-dkol}, the exponent $x$ equals $CS(\delta,\alpha)=1-\frac{\min(3,\delta)}2 + \alpha\left(\frac 92\min(3,\delta)-5\right)$. 
Note that Figure \ref{fig:convergencespeedplot} consists in the above functions for $\alpha=0$. 
In particular, we have 
\begin{equation*} 
CS(\delta,\alpha)<JL(\delta,\alpha)\iff \alpha< 
\begin{cases}\displaystyle \frac{\delta ^2-2 \delta }{9 \delta ^2-2 \delta -8} & \text{ if }\delta <3 \\\displaystyle \frac{3}{16 \delta +19} & \text{ if }\delta \geq 3\end{cases} 
\end{equation*} 
and 
\begin{equation*} 
P(\delta,\alpha)<JL(\delta,\alpha)\iff \delta\in[4,5) \text{ and } \alpha<\frac{5-\delta}{8+2\delta} 
\end{equation*} 
and finally 
\begin{equation*} 
P(\delta,\alpha)<CS(\delta,\alpha)\iff\delta \geq 4 \text{ and } \alpha>\frac 1{30}. 
\end{equation*}

Note that we considered the case where the $(Y_k)_{k\in V}$ had ``homogeneous’’ distributions. The presence of $\kA_\delta$ in the bounds of Theorems \ref{thm:mod-phi-corollary-new}, \ref{thm:dkol-delta-3+} and \ref{thm:result-refined-delta<2} allows to establish Berry--Esseen bounds even for highly non-identically distributed random variables, for which the above comparison would be even more favorable to our results.

\begin{figure}[ht] 
	\centering 
	\includegraphics[width=0.9\linewidth]{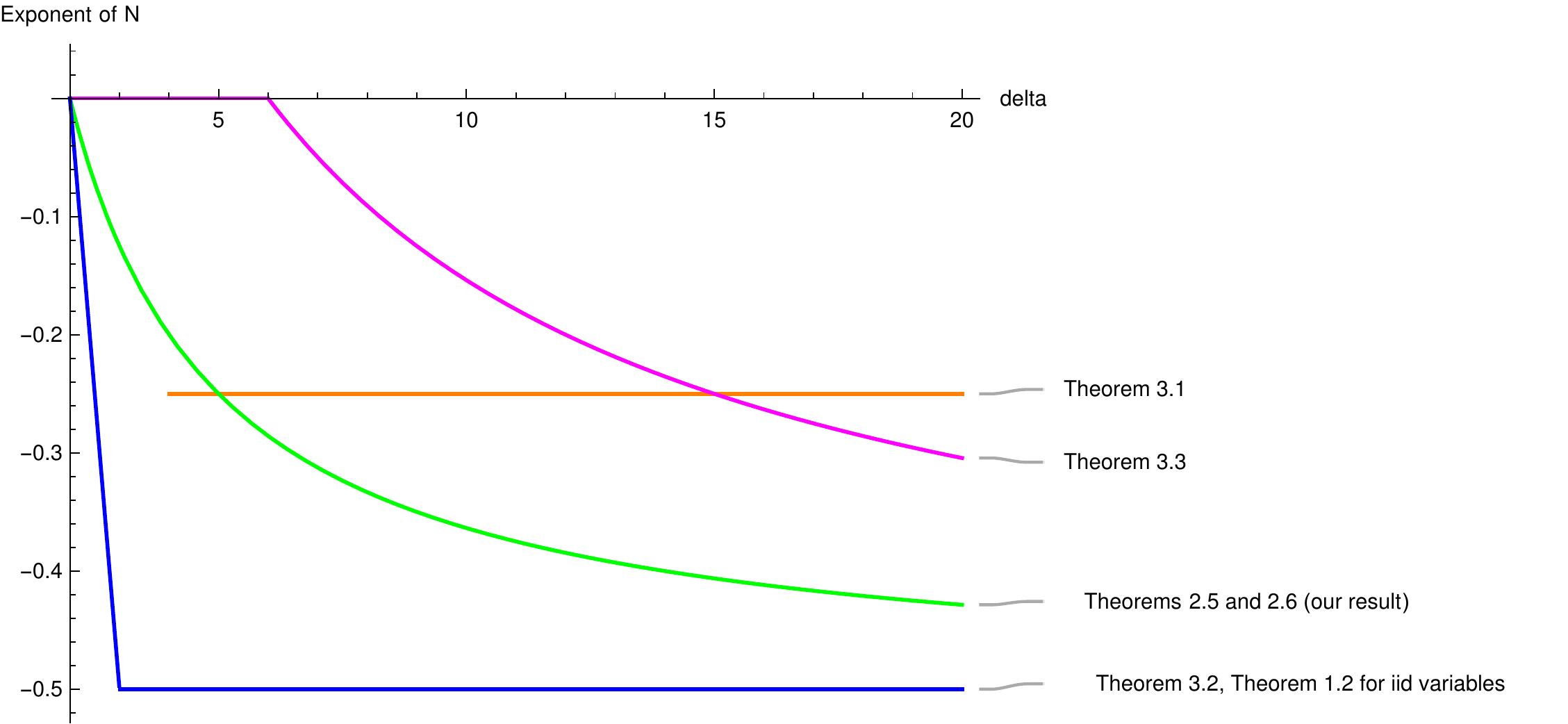} 
	\caption{Comparison of the exponent in $N$ in Theorems 
	\ref{thm:dkol-delta-3+} and \ref{thm:result-refined-delta<2} (our results, in green), 
	\ref{thm:Stein-result} (orange), 
	\ref{thm:Stein-dkol} (blue) and 
	\ref{thm:modII-result} (pink). 
	 The lower the line, the better. The green line is always below the pink line, reflecting the fact that our results are strictly sharper than Theorem \ref{thm:modII-result}. The exponent of $N$ result is strictly smaller than that of Theorem \ref{thm:Stein-result} if and only if $2<\delta< 4$ or $\delta>5$. For comparison, Berry--Esseen for identically distributed and independent random variables is marked as red and cannot be improved upon, cf. \cite{Petrov, Shevtsova2020}; the same exponent is attained in Theorem \ref{thm:Stein-dkol}, however, they pay the price of a much larger power of the maximal degree of the dependency graph. 
	} 
	\label{fig:convergencespeedplot} 
\end{figure} 
 
\begin{figure}[ht] 
	\centering 
	\includegraphics[width=0.9\linewidth]{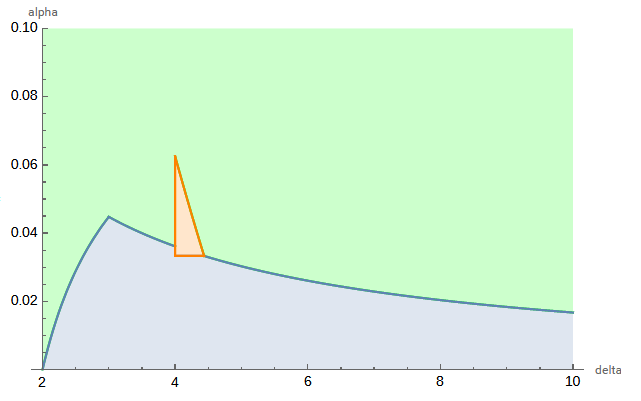} 
	\caption{We display the regions of the set of possible $(\delta,\alpha)$ for which the best bound among $JL(\delta,\alpha)$, $CS(\delta,\alpha)$ and $P(\delta,\alpha)$ is $JL(\delta,\alpha)$ (Theorems \ref{thm:dkol-delta-3+} and \ref{thm:result-refined-delta<2}, green), $P(\delta,\alpha)$ (Theorem \ref{thm:Stein-result}, orange) and $CS(\delta,\alpha)$ (Theorem \ref{thm:Stein-dkol}, blue), as described at the end of section \ref{sec:direct-comparison}. It should be noted that we only consider the region $\alpha\in[0,0.1]$ since the region $\alpha\in[0.1,1]$ is fully green. 
	} 
	\label{fig:comparison-D-N} 
\end{figure}

\section{Optimality}\label{sec:optimality}

In the previous section, we based our comparison on the example where the $(Y_k)_{k\in V}$ were grouped into cliques of $D+1$ of r.v. with maximal correlation inside the cliques and independence between cliques. This is the case where we expect $\dkol(W, \mathcal{N}(0,1))$ to be largest. This example can thus give us insight on how close to optimal our bound is. In the i.i.d. case, the known optimality of the Berry--Esseen bound ensures that for every $C_- < \frac{\sqrt{10}+3}{6\sqrt{2\pi}} \approx 0.4097$ (we fix one for the rest of this section), there exist bounded i.i.d. random variables $(Z_k)_{k\geq 0}$ with unit variance and zero expectation such that (cf. \cite{Esseen56}) 
\begin{align}\label{eq:rev-BE} 
	\dkol\left(\frac{Z_0 + \dots + Z_{n-1}}{\sqrt{n}} , \kN(0,1)\right) > C_- \frac{\E[|Z_0|^3]}{\V[Z_0]^{3/2}} \frac{1}{\sqrt{n}}. 
\end{align} 
As before, 
consider the family $(Y_k)_{1\leq k \leq N}$, where $N = n(D+1)$, and for every $0 \leq k \leq n-1$, 
$$ Y_{k(D+1)+1} = Y_{k(D+1)+2} = \dots = Y_{k(D+1)+D+1} = Z_{k}. $$ 
Then 
$$ S = (D+1) \sum_{k=0}^{n-1} Z_k , $$ 
with $\V[(D+1)Y_{1}] = (D+1)^2 \V[Z_0]$ and $\E[|(D+1)Y_1|^3] = (D+1)^3 \E[|Z_0|^3]$. Applying the ``reverse Berry--Esseen bound'' \eqref{eq:rev-BE} to $S$ gives 
\begin{align*} 
	\dkol\left(W , \kN(0,1)\right) > C_- \frac{\E[|Z_0|^3]}{\V[Z_0]^{3/2}} \sqrt{\frac{D+1}{N}} \ . 
\end{align*} 
On the other hand, Theorem \ref{thm:mod-phi-corollary-new} gives that there exists a universal constant $C'$ such that 
\begin{equation*} 
	\dkol\left(W , \kN(0,1)\right) \leq C' \sqrt{\frac{D+1}N}\max\left\{\frac{\E[\abs{Z_0}^3]}{\V[Z_0]^\frac{3}{2}},  \frac{\|Z_0\|_\infty}{\V[Z_0]}\right\} . 
\end{equation*} 
Theorem \ref{thm:mod-phi-corollary-new} thus gives a matching bound up to a multiplicative constant, meaning that it is, in its full generality, almost optimal. 
 
More generally, for $(Y_k)_{k\in V}$ with $\max_{k\in V}\|Y_k\|_\infty \leq L$, as long as $\V[S]\ge \mathfrak c (D+1) N$ for some $\mathfrak c>0$ independent of $N$ and $D$, 
Theorem \ref{thm:mod-phi-corollary-new} implies 
\begin{equation*} 
	\dkol\left(W, \kN(0,1)\right)\lesssim \sqrt{\frac{D+1}N} \max\left\{\frac{\kA_3}{N}, L\right\}. 
\end{equation*} 
This is as good as \eqref{eq:Berry-Esseen-classical} up to multiplicative constant when substituting $1/N$ by $(D+1)/N$, which we expect to be the correct surrogate. There is not much room for improvement. However, in Theorem \ref{thm:dkol-delta-3+} and \ref{thm:result-refined-delta<2} where the $Y_k$ are not uniformly bounded in $L^\infty$, we expect that our exponents can be improved upon, see Conjecture \ref{conj:main}. 
Since the exponents of $N$ and $D$ play a much larger role for the quality of the bound than the numerical constant in front of them, it is thus of little interest to find the best possible numerical constant.

We can now explain Conjecture \ref{conj:main}. The optimal exponent of $N$ is already known. Indeed, Theorem \ref{thm:Stein-dkol} ensures that, for $D$ fixed, the exponent of $N$ is that of the Berry--Esseen inequality for independent random variables. The exponent of $D$ in Theorem \ref{thm:Stein-dkol} is however not optimal, as is demonstrated by the comparison in Section \ref{sec:direct-comparison}. To obtain the Conjecture, we thus use the heuristic of substituting $N$ with $(D+1)/N$. This heuristic is underpinned by the following two observations: The relevant quantity in the Theorems of Section \ref{sec:main-results} is $(D+1)/N$ instead of $N$; and the example above gives the same results as the i.i.d. case with $N$ replaced by $(D+1)/N$. 
 
\begin{remark} 
	In the independent case, Theorem \ref{thm:dkol-delta-3+} gives as Corollary a result which is similar to the classical Lyapunov Central Limit Theorem. 
	Indeed, recall that for a sequence $(Y_k)_{k\in\N}$ of independent random variables, for which we assume without loss of generality that $\E[Y_k]=0$ for every $k$, and letting $S_N = \sum_{k=1}^N Y_k$, the \emph{Lyapunov condition} 
	\begin{equation}\label{eq:Lyapunov-condition} 
		\lim_{N\to\infty} \frac{1}{\V[S_N]^\frac\delta2} \sum_{k=1}^N \E[\abs{Y_k}^\delta] = 0\quad\text{for some }\delta>2 
	\end{equation} 
	implies $S_N/\sqrt{\V[S_N]} \to\kN(0,1)$ in distribution as $n\to\infty$. 
	Assuming that there are some $0<a<A$ such that $\norm{Y_k}_\delta \in (a,A)$ for every $k$, \eqref{eq:Lyapunov-condition} is equivalent to 
	\begin{equation}\label{eq:Lyapunov-in-disguise} 
		\limsup_{N\to\infty}\frac{N^{\frac 2\delta}}{\V[S_N]} = 0 \ . 
	\end{equation} 
	This is precisely the condition which ensures that the right-hand side of \eqref{eq:dkol-3+-sigma-large} converges to $0$ as $N\to\infty$. 
	Indeed, the statement \begin{equation}\label{eq:convergence-to-0-for-RHS-of-our-THM}\lim_{N\to\infty}\xi_\delta^{-\frac{\delta}{\delta+1}} \left(\frac{D+1}N\right)^{\frac{\delta-2}{2(\delta+1)}}=0 \end{equation} is equivalent to $\lim_{N\to\infty}\xi_\delta^{-\delta} N^{\frac{2-\delta}2}=0$, which for $D=0$ reduces to $\lim_{N\to\infty}  \frac{\kA_\delta}{N} \frac{N}{\V[S]^{\frac\delta2}}=0$. By assumption, $\frac{\kA_\delta}N\in[a^\delta, A^\delta]$, so that \eqref{eq:Lyapunov-in-disguise} implies \eqref{eq:convergence-to-0-for-RHS-of-our-THM}. Similarly (details left to the interested reader), one can show that \eqref{eq:Lyapunov-in-disguise} implies $\lim_{N\to\infty}\frac{1}{\xi_3^3}\sqrt{\frac{D+1}N}=0$. 
\end{remark}

\begin{remark}\label{rem:power-of-sigma-is-optimal} 
	The power for $\xi_\delta$ in \eqref{eq:dkol-3+-sigma-large} is optimal in the sense that for any $\delta\geq2$, multiplying the right-hand-side of \eqref{eq:dkol-3+-sigma-large} by $g(\xi_\delta)$ for any function $g$ that converges to $0$ as $\xi_\delta\to 0$ makes the inequality wrong. It does not mean that the right-hand side of \eqref{eq:dkol-3+-sigma-large} is optimal, but that a different exponent on $N$ and on $D$ would be required in order to improve the bound. 
	To see that the exponent of $\xi_\delta$ is optimal, let $\delta\ge 3$ and let $(Y_k)_{k\ge 2}$ be an independent sequence of random variables such that 
	\begin{equation*} 
		\P[Y_k=k^\frac1\delta]=\P[Y_k=-k^\frac1\delta]=\frac{1-\P[Y_k=0]}2=\frac{k^\frac2\delta-(k-1)^\frac2\delta}{2 k^\frac2\delta}. 
	\end{equation*} 
	By the Mean Value Theorem, for every $k\ge 2$, there exists $\zeta(k)\in[k-1, k]$ such that $k^{\frac 2\delta}- (k-1)^{\frac 2\delta} = \frac2\delta \zeta(k)^{\frac 2\delta-1}$. In particular, $k^{\frac 2\delta}- (k-1)^{\frac 2\delta}\le k^{\frac2\delta-1}$. 
	Then $\E[Y_k]=0$, $\E[\abs{Y_k}^\delta] = k^{1-\frac2\delta} (k^\frac2\delta-(k-1)^\frac2\delta)\le 1$ and $\V[Y_k] = k^{\frac 2\delta}-(k-1)^{\frac 2\delta}$. In particular, $\E[S_N] = 0$ and $\V[S_N] = N^{\frac 2\delta}$ for every $N\ge 2$. 
	 
	We claim that $(S_N-\E[S_N])/\sqrt{\V[S_N]}$ does not converge in distribution to $\kN(0,1)$ as $N\to\infty$. 
	To prove this, note that the Feller condition $\lim_{N\to\infty} \max_{1\le k \le N} \V[Y_k] / \V[S_N]=0$ is satisfied, because $\V[Y_k]$ is uniformly bounded in $k$ while $\V[S_N]\to\infty$ as $N\to\infty$. On the other hand, the Lindeberg condition 
	\begin{equation}\label{eq:Lindeberg-condition} 
		\lim_{N\to\infty} \frac{1}{\V[S_N]} \sum_{k=1}^N \E[Y_k^2 [Y_k^2 > \varepsilon^2 \V[S_N]]] = 0 \quad\text{for every }\varepsilon>0 
	\end{equation} 
	is not satisfied. Indeed, pick $\varepsilon=\frac12$. Then $Y_k^2 [Y_k^2 > \varepsilon^2 \V[S_N]]\neq 0 \iff k^{\frac2\delta}> \frac{N^\frac2\delta}2\iff k>N 2^{-\frac\delta 2}$. Let $N_\dagger$ denote the smallest integer strictly greater than $N 2^{-\frac\delta 2}$. Then 
	\begin{equation*} 
		\frac{1}{\V[S_N]} \sum_{k=1}^N \E[Y_k^2 [Y_k^2 > \varepsilon^2 \V[S_N]]] = N^{-\frac 2\delta} \left(N^{\frac 2\delta} - N_\dagger^{\frac 2\delta}\right)=1-\left(\frac{N_\dagger}{N}\right)^{\frac2\delta}. 
	\end{equation*} 
	But $\lim_{N\to\infty} \frac{N_\dagger}N = 2^{-\frac\delta 2}$ and therefore \eqref{eq:Lindeberg-condition} cannot be satisfied. By the Lindeberg Central Limit Theorem, if the Feller condition is satisfied, then $(S_N-\E[S_N])/\sqrt{\V[S_N]}\to \kN(0,1)$ in distribution as $N\to\infty$ if and only if the Lindeberg condition is satisfied. We conclude that $\dkol((S_N-\E[S_N])/\sqrt{\V[S_N]} , \kN(0,1))\not\to 0$ as $N\to\infty$. 
	 
	To see why the exponent of $\xi_\delta$ in \eqref{eq:dkol-3+-sigma-large} is optimal, observe that we can apply Theorem \ref{thm:dkol-delta-3+} with $D=0$, $A=2$, and $\xi_\delta = \cste \cdot N^{\frac{1}{\delta}-\frac{1}{2}}$. The right-hand side of \eqref{eq:dkol-3+-sigma-large} is thus a positive constant: we cannot multiply the right-hand side of \eqref{eq:dkol-3+-sigma-large} by any $g(\xi_\delta)$ with $g(\xi_\delta)\to 0$ as $\xi_\delta\to 0$, because it would contradict the previous observation. 
	 
	The same example also shows that the power of $\xi_\delta$ in \eqref{eq:result-refined-delta<2} is optimal. 
\end{remark}

\section{Applications} 
\label{sec:applications}

\subsection{Central Limit Theorems and Weak Laws of Large Numbers}
\label{sec:CLTs} 
 
To further illustrate our results, in this section we prove some Central Limit Theorems and Weak Laws of Large Numbers that follow from our results. 
 
\begin{corollary}\label{cor:CLT and WLLN} 
	Consider a triangular array  $(Y_{N,k})_{N\in\N, k\in\set{1,\dots,N}}$ such that for every $N$,\\ $(Y_{N, k})_{k\in\set{1,\dots,N}}$ has a dependency graph of maximum degree $D(N)$. Fix a real-valued family $(c_{N,k})_{N\in\N, k\in\{1,\dots,N\}}$. Define $S_N = \sum_{k=1}^N Y_{N,k}$, $v_N = \sqrt{\V[S_N]}$, $W_N = (S_N - \E[S_N])/v_N$, $\kA_\delta(N) = \sum_{k=1}^N \E[|Y_{N,k}-c_{N,k}|^\delta]$ and 
	$$\xi_\delta(N) = \left(\frac{N}{\kA_\delta(N)}\right)^{\frac 1\delta} \sqrt{\frac{\V[S_N]}{N(D(N)+1)}} \ . $$ 
	Assume that $v(N)>0$ for every $N$ large enough. If 
	\begin{equation}\label{eq:finite-delta-CLT-condition} 
	\delta\in(2,3)\quad\text{ and }\quad 
	\frac 1{\xi_\delta(N)} \left(\frac {D(N)+1}{N}\right)^{\frac 12-\frac 1\delta} \ulim N \infty 0 
	\end{equation} 
	or 
	\begin{equation}\label{eq:finite-delta-CLT-condition>3} 
	\delta\in[3,\infty)\quad\text{ and }\quad 
	 \xi_\delta(N)^{-1} \left(\frac {D(N)+1}N\right)^{\frac 12-\frac 1\delta} \ulim N \infty 0 \quad\text{ and }\quad 	\xi_\delta(N)^{-3} \sqrt{\frac{D(N)+1}{N}} \ulim N \infty 0 
	\end{equation} 
	or 
	\begin{equation}\label{eq:infinite-delta-CLT-condition} 
	\sup_{k\in\N}\norm{Y_{N,k}-c_{N,k}}_\infty<\infty\quad\text{ and }\quad  \frac{D(N)+1}{v_N} \ulim N \infty 0 \quad\text{ and }\quad   \frac{N(D(N)+1)^2}{v_N^3} \ulim N \infty 0, 
	\end{equation} 
	then 
	\begin{equation}\label{eq:CLT} 
	W_N \to\kN(0,1) 
	\end{equation} 
	in distribution as $N\to\infty$. 
	 
	In particular, we also obtain the Weak Law of Large Numbers, i.e. 
	\begin{equation*} 
	\frac{S_N-\E[S_N]}{\V[S_N]}\to 0 
	\end{equation*} 
	in probability as $N\to\infty$ whenever either \eqref{eq:infinite-delta-CLT-condition} holds, or $\liminf_{N\to\infty}\kA_\delta(N)>0$ and either \eqref{eq:finite-delta-CLT-condition} or \eqref{eq:finite-delta-CLT-condition>3} holds. 
\end{corollary} 
 
\begin{proof}[Proof of Corollary \ref{cor:CLT and WLLN}] 
	\eqref{eq:finite-delta-CLT-condition}, \eqref{eq:finite-delta-CLT-condition>3} or \eqref{eq:infinite-delta-CLT-condition} ensures that the right-hand side of Theorem \ref{thm:result-refined-delta<2}, \ref{thm:dkol-delta-3+} or \ref{thm:mod-phi-corollary-new} converges to $0$ as $N\to\infty$, which proves \eqref{eq:CLT}. 
	 
	The first step towards the Weak Law of Large Numbers is to check that $v_N\to\infty$. This is implied by the last condition of \eqref{eq:infinite-delta-CLT-condition}. In the other cases, 
	\begin{equation*} 
		\frac 1 {\xi_\delta(N)} \left(\frac{D(N)+1}N\right)^{\frac 12-\frac 1\delta} = \frac 1{v_N} \left(\frac{\kA_\delta(N)}N\right)^{\frac 1\delta} N^{\frac 1\delta} (D(N)+1)^{1-\frac 1\delta}\ge \frac 1{v_N}\kA_\delta(N)^{\frac 1\delta}. 
	\end{equation*} 
	thus $\kA_\delta(N)^{1/\delta}/v_N \to 0$ as $N\to\infty$ under \eqref{eq:finite-delta-CLT-condition} or \eqref{eq:finite-delta-CLT-condition>3}. 
	Together with $\liminf_{N\to\infty} \kA_\delta(N)>0$, it implies $v_N \to \infty$ as $N\to\infty$. 
	By Slutzky's Theorem together with \eqref{eq:CLT} we have convergence in distribution of $\frac{S_N-\E[S_N]}{\V[S_N]}$ to $0$. We conclude by observing that convergence in distribution to a constant implies convergence in probability. 
\end{proof} 
 
\begin{remark} 
	Since $\xi_\delta(N) \le 1$ for every $\delta\in(2,\infty)$, conditions \ref{eq:finite-delta-CLT-condition} and \ref{eq:finite-delta-CLT-condition>3} imply in particular $\frac{D(N)+1}N\to 0$ as $N\to\infty$. Furthermore, $v_N^2\le N (D(N)+1) (\sup_{k\in\{1,\dots,N\}}\norm{Y_{N,k}-c_{N,k}}_\infty)^2$, so that condition \ref{eq:infinite-delta-CLT-condition} also implies $\frac{D(N)+1}N\to0$ as $N\to\infty$. 
\end{remark} 
 
\begin{remark} 
	To illustrate the sharpness of condition \ref{eq:infinite-delta-CLT-condition}, consider for example $(X_n)_{n\in\Z_{\ge 0}}$ i.i.d. with 
	\begin{equation*} 
	\P[X_0=1]=\P[X_0=-1]=\frac 12. 
	\end{equation*} 
	In particular, $\E[X_0]=0, \V[X_0]=1$. 
	Fix a function $f:\N\to\N$ (for example $f(N)=\floor{N^{2/3}}$) such that $f(N)\le N$ for every $N\in\N$ and 
	\begin{equation*}\lim_{N\to\infty}\frac{f(N)}N=\lim_{N\to\infty} \frac{N}{f(N)^2} = 0.\end{equation*} 
	Consider the triangular array of random variables $(Y_{N, k})_{N\in\N, k\in\set{1,\dots,N}}$ given by $Y_{N,k}=X_0$ if $k\le f(N)$ and $Y_{N,k}=X_{k-f(N)}$ if $k>f(N)$. 
	For every fixed $N\in\N$, $(Y_{N,k})_{k\in\set{1,\dots,N}}$ thus admits a dependency graph such that $D+1=f(N)$. Furthermore, $\sup_{N,k}\norm{Y_{N,k}}_{L^\infty}=1<\infty$. 
	Let $S_N=\sum_{k=1}^N Y_{N,k}$. The variance of $S_N$ is 
	\begin{equation*} 
		\V[S_N]=f(N)^2\V[X_0] + (N-f(N))\V[X_1]=f(N)^2+N-f(N). 
	\end{equation*} 
	By the classical Central Limit Theorem, 
	\begin{equation*} 
		\frac{1}{\sqrt{N-f(N)}}\sum_{k=f(N)+1}^N Y_{N,k}\to\kN(0,1) \text{ in distribution as }N\to\infty. 
	\end{equation*} 
	Since $f(N)^2/N \to \infty$ as $N\to\infty$, 
	\begin{equation*} 
		\frac{1}{\sqrt{f(N)^2+N-f(N)}} \sum_{k=f(N)+1}^N Y_{N,k}\to 0 \text{ in distribution as }N\to\infty. 
	\end{equation*} 
	Furthermore, 
	\begin{equation*} 
		\frac{1}{\sqrt{f(N)^2+N-f(N)}}\sum_{k=1}^{f(N)} Y_{N,k} = \frac{f(N)}{\sqrt{f(N)^2+N-f(N)}} X_0\to X_0\text{ almost surely as }N\to\infty. 
	\end{equation*} 
	Hence, 
	\begin{equation*} 
		\frac{S_N}{\sqrt{\V[S_N]}}\to X_0\text{ in distribution as }N\to\infty, 
	\end{equation*} 
	but $X_0$ is not $\kN(0,1)$ distributed. The reason that Corollary \ref{cor:CLT and WLLN} doesn't apply here is that $\lim_{N\to\infty}\frac{D+1}v=\lim_{N\to\infty}\frac{f(N)}{\sqrt{f(N)^2+N-f(N)}}=1\neq 0$. 
\end{remark}

\subsection{Application to \texorpdfstring{$U$}{U}-statistics}\label{sec:U-statistics}

Throughout the rest of the text, for every $a,b\in\R$, $k\in\N$ and every set $A$, we write $\llbracket a,b\rrbracket\define\Z\cap[a,b]$ and $\binom Ak\define\set{B\subset A:\abs{B}=k}$.

We apply Theorems \ref{thm:mod-phi-corollary-new}, \ref{thm:dkol-delta-3+} and \ref{thm:result-refined-delta<2} to \emph{$U$-statistics}, a class of statistics that is especially important in estimation theory. 
 
\begin{definition} 
	Let $\kS$ be a measurable space, $\ell\in\N$, $(X_n)_{n\in\N}$ a sequence of $\kS$-valued random variables and $f:\kS^\ell\to\R$ a measurable function. The \emph{$U$-statistics} of $f$ are defined, for each $n\in\N$, as 
	\begin{equation}\label{eq:def-U-stat} 
		U_n=U_n(f)=U_n(f, (X_n)_{n\in\N})=\frac{1}{\ell!\binom nl}\sum_{\alpha\in\Lambda_{n, \ell}} f(X_\alpha), 
	\end{equation} 
	where 
	\begin{equation*} 
		\Lambda_{n,\ell}\define\set*{(\alpha_1,\dots,\alpha_\ell)\in\llbracket 1,n\rrbracket^\ell: \set{\alpha_1,\dots,\alpha_\ell}\in\binom{\llbracket 1,n\rrbracket}{\ell}} 
	\end{equation*} 
	and 
	\begin{equation*}	 
		X_\alpha\define(X_{\alpha_1},\dots,X_{\alpha_\ell}) \qquad \text{ for } \qquad \alpha=(\alpha_1, \dots, \alpha_\ell)\in\Lambda_{n,\ell}. 
	\end{equation*} 
\end{definition} 
 
\begin{examples*} 
 
	\begin{enumerate} 
		\item If $\kS=\R, \ell=1$ and $f:\R\to\R,x\mapsto x$, the U-statistics are the \emph{sample means} 
		\begin{equation}\label{eq:sample-mean} 
			U_n = \frac{X_1+\dots+X_n}n \define \bar X_n . 
		\end{equation} 
		\item If $\kS=\R, \ell=2$ and $f:\R^2\to\R, (x,y)\mapsto (x-y)^2/2$, the U-statistics are the \emph{sample variances} 
		\begin{equation}\label{eq:sample-variance}\begin{split} 
				U_n &= \frac 1{4\binom n2}\sum_{i,j=1}^n (X_i-X_j)^2 = \frac{1}{n-1}\sum_{k=1}^n (X_k-\bar X_n)^2 \define \hat\Sigma_n . 
		\end{split}\end{equation} 
	\end{enumerate} 
\end{examples*} 
 
When using U-statistics as estimators, a question of importance is whether they are consistent, and if they are, at what speed they converge to the quantity that they estimate. A satisfying answer to this question is usually given in the form of a Central Limit Theorem. 
Janson \cite{JansonU2021} proved a Central Limit Theorem when the $(X_n)_{n\in\N}$ are \emph{$m$-dependent}: 
 
\begin{definition}\label{def:m-dependent} 
	Let $V\subset\Z$ and $m\in\Z_{\ge 0}$. A family of random variables $(X_n)_{n\in V}$ is called \emph{$m$-dependent} if and only if, for every $k\in\N$, $(X_{n})_{n\le k}$ and $(X_n)_{n\ge k+m+1}$ are independent. 
\end{definition} 
 
Independence is equivalent to $0$-dependence. A $m$-dependent family of random variables always admits a dependency graph with maximum degree at most $2m$, namely the graph in which $i,j\in V$ are connected if and only if $1\le\abs{i-j}\le m$. 
This graph has maximal degree less or equal than $2m$. On the other hand, if $(X_n)_{n\in V}$, with $V\subset\N$ finite, is a family of random variables admitting a dependency graph of maximal degree less or equal than $2m$, there need not exist a permutation $\sigma:V\to V$ such that $(X_{\sigma(n)})_{n\in V}$ is $m$-dependent.   
 
\begin{theorem}[{\cite[Theorem 3.8]{JansonU2021}}]\label{thm:Janson} 
	Let $m\in\Z_{\ge 0}$ and let $(X_n)_{n\in\N}$ be a stationary sequence of $m$-dependent random variables with values in a measurable space $\kS$. Let $l\in\N$ and let $f:\kS^l\to\R$ satisfy 
	\begin{equation*} 
		f(X_\alpha)\in L^2\quad\text{ for every } \alpha\in\Lambda_{n,\ell}. 
	\end{equation*} 
	Let $V_n \define l! \binom nl U_n=\sum_{\alpha\in\Lambda_{n,\ell}} f(X_\alpha)$ denote the \enquote{non-normalized U-statistic}. Then there exists a $\sigma^2\in[0,\infty)$ such that 
	\begin{equation}\label{eq:convergence-of-variances-of-Vn} 
		\lim_{n\to\infty}\frac{\V[V_n]}{n^{2l-1}}=\sigma^2, 
	\end{equation} 
	and 
	\begin{equation}\label{eq:convergence-result-Janson-form} 
		\frac{V_n-\E[V_n]}{n^{l-\frac 12}}\to\kN(0, \sigma^2)\quad\text{in distribution as }n\to\infty. 
	\end{equation} 
\end{theorem} 
 
\begin{remark} 
	Using that $\frac{\Gamma(x+1)x^\varepsilon}{\Gamma(x+1+\varepsilon)}=1+O(1/x)$ as $x\to\infty$ for every fixed $\varepsilon>0$, cf. \cite[8.328, 2.]{Gradshteyn-Ryzhik}, 
	 $$\frac{\V[V_n]}{\V[U_n]} = \left(\frac{n!}{(n-l)!}\right)^2 = n^{2l} (1+O(1/n))$$ 
	  as $n\to\infty$. Therefore, 
	\eqref{eq:convergence-of-variances-of-Vn} 
	is equivalent to 
	\begin{equation*} 
		\lim_{n\to\infty} n \V[U_n] = \sigma^2. 
	\end{equation*} 
	and \eqref{eq:convergence-result-Janson-form} 
	is equivalent to 
	\begin{equation*} 
		\frac{U_n-\E[U_n]}{\sqrt{\V[U_n]}}\to\kN(0,1) \quad\text{in distribution as }n\to\infty. 
	\end{equation*} 
\end{remark}

Under additional assumptions, we can supplement Theorem \ref{thm:Janson} with an estimate of the speed at which $(V_n-\E[V_n])/\sqrt{\V[V_n]}$ converges towards $\mathcal{N}(0,1)$. The idea, which is used both in \cite[Section 9]{JansonU2021} and in Corollary \ref{cor:U-stats} below, relies on writing the U-statistic as a sum of the r.v. $(f(X_\alpha))_{\alpha\in\Lambda_{n,\ell}}$, which have an explicit dependency graph. Every Berry--Esseen-type estimate on r.v. with a dependency graph then directly yields an estimate on the speed of convergence of the U-statistic.

Let us describe the dependency graph. For generality we consider $(X_1, \dots, X_n)$ with a dependency graph of maximum degree $m$ --- this includes the case of $m$-dependence. Consider the graph on the vertex set $V=\Lambda_{n,\ell}$ such that $\alpha=(\alpha_1,\dots,\alpha_\ell), \beta=(\beta_1,\dots,\beta_\ell)\in V$ are connected if and only if there exist $i,j\in\llbracket 1,\ell\rrbracket$ such that either $\alpha_i = \beta_j$, or $\alpha_i, \beta_j$ are connected in the dependency graph of the $(X_1,\dots, X_n)$. This graph has cardinality $N\define\abs{V}=\binom nl \ell! = \frac{n!}{(n-\ell)!}$ and maximal degree 
	\begin{equation}\label{eq:bound-D-U} 
	D\le \ell^2 (m+1)\binom{n-1}{\ell-1} (\ell-1)! - 1 = \frac{\ell^2(m+1)}{n} N -1 . 
	\end{equation} 
	Indeed, for any fixed $\alpha$, to construct any adjacent $\beta$ we first need to choose at least one pair of coordinates $(i,k) \in\set{1,\dots, \ell}^2$ and $\beta_k$ such that $\alpha_i$ and $\beta_k$ are connected in the dependency graph of $(X_1,\dots, X_n)$ or $\alpha_i = \beta_k$. For this, there are at most $\ell^2 (m+1)$ possibilities. Then one has to choose the other $\ell-1$ coordinates, for which there are $\binom{n-1}{\ell-1}(\ell-1)!$ possibilities, since order matters. Finally one has to subtract $1$ because $\alpha$ is not adjacent to itself.

This idea applies in more general situations than considered in \cite{JansonU2021}. Constrained, resp. exactly constrained U-statistics \cite[Section 3]{JansonU2021} differ only in the set of all $\alpha$ such that $f(X_\alpha)$ appears in $U_n$, as well as in the dependency graph that they induce on the family $(f(X_\alpha))_{\alpha}$. Furthermore, we need not assume that the $(X_n)_{n \in \N}$ be $m$-dependent as in \cite{JansonU2021}: it suffices that they have a dependency graph of maximal degree $2m$. We also do not assume that the $(X_n)_{n\in \N}$ are stationary --- this is used by Janson to determine the variance of the U-statistic, something that we do not consider. Finally, it is possible to consider the case where $m$ depends on $n$, whereas $m$ is fixed in  \cite{JansonU2021}.

\begin{corollary}\label{cor:U-stats-bounded} 
	Let $(X_k)_{1\leq k \leq n}$ be random variables with values in a measurable space $\kS$ that admit a dependency graph of maximal degree $m$. Let $(c_\alpha)_{\alpha\in\Lambda_{n,\ell}}$ be an arbitrary family of real numbers, $\ell\in\N$, and $f:\kS^\ell\to\R$ measurable such that $\max_{\alpha\in\Lambda_{n,\ell}}\norm{f(X_\alpha)-c_\alpha}_{\infty}\define L \in(0,\infty)$. 
	Recall $U_n$ defined by \eqref{eq:def-U-stat} and $V_n\define l! \binom nl U_n$ as in Theorem \ref{thm:Janson}. 
	Define furthermore $\mathbf a_\delta\define \left( \kA_\delta / \abs{\Lambda_{n,\ell}}\right)^{1/\delta}=\left(\kA_\delta (n-\ell)!/n!\right)^{1/\delta}$. 
	Then 
	\begin{equation}\label{eq:U-stats-4} 
		\dkol\left(\frac{U_n-\E[U_n]}{\sqrt{\V[U_n]}}, \kN(0,1)\right)\le 227.5 \sqrt{\frac{\ell^2(m+1)}n} \Xi_\infty^{-3}, 
	\end{equation} 
	where 
	\begin{equation*} 
		\Xi_\infty\define\frac 1L \sqrt{\frac{\V[V_n]}{\ell^2(m+1)n^{2l-1}}}\in(0,1]. 
	\end{equation*} 
\end{corollary} 
 
The proof is a straightforward application of Theorem \ref{thm:mod-phi-corollary-new}. We refer the reader to the proof of Corollary \ref{cor:U-stats} which follows the same structure.

Corollary \ref{cor:U-stats-bounded} is similar to \cite[Theorem 3.14]{JansonU2021}. 
Let us compare both. 
While Corollary \ref{cor:U-stats-bounded} is formulated in a more general setting than \cite[Theorem 3.14]{JansonU2021} (for example the $(X_n)_{n\in \N}$ are not assumed to be $m$-dependent), since the ideas used in the proof of both theorems are the same, it takes little effort to generalize \cite[Theorem 3.14]{JansonU2021} to the same setting. For the same reason, the proof of \cite[Theorem 3.14]{JansonU2021} can also yield explicit, nonasymptotic estimates like \eqref{eq:U-stats-4}. 
The additional requirements in 
\cite[Theorem 3.14]{JansonU2021} that $m$ be fixed, that $(X_n)_{n\in\N}$ be stationary and that $\liminf_{n\to\infty}\frac{\V[U_n]}{n^{2l-1}}>0$ guarantee that the right-hand side of \eqref{eq:U-stats-4} converges to $0$: Indeed, for any fixed $\ell$, whenever $m+1=o(n^\frac 14)$ and $\liminf_{n\to\infty}\frac{\V[V_n]}{n^{2l-1}}>0$, the right-hand side of \eqref{eq:U-stats-4} converges to $0$. 
 
The most significant difference is that \cite[Theorem 3.14]{JansonU2021} uses \cite[Theorem 2.2]{Rinott} while we use the slightly better Theorem \ref{thm:mod-phi-corollary-new}. Some other results that can be used are listed in \cite[Remark 9.1]{JansonU2021}; they require at least that $\sup_\alpha \E[|f(X_\alpha)|^4] < \infty$. We are able to relax this assumption in Corollary \ref{cor:U-stats} by using our results, Theorems \ref{thm:mod-phi-corollary-new}, \ref{thm:dkol-delta-3+} and \ref{thm:result-refined-delta<2}: see Corollary \ref{cor:U-stats} below. 
 
\begin{remark} 
Theorem \ref{thm:Stein-result} and Theorem \ref{thm:Stein-dkol} give trivial bounds when used on the family $(f(X_\alpha))_{\alpha\in\Lambda_{n,\ell}}$. Indeed, the degree of the dependency graph is larger than $N^{1/2}/100$ for $l\geq 2$ and $n$ large enough. Bounds with a better dependency on the degree of the dependency graph, like ours, are required. 
\end{remark}

\begin{corollary}\label{cor:U-stats} 
	Let $(X_k)_{1\le k\le n}$, $\ell$, $(c_\alpha)_{\alpha \in \Lambda_{n,\ell}}$ be as in Corollary \ref{cor:U-stats-bounded}. Let $f:\kS^\ell\to\R$ measurable such that $\sum_{\alpha\in\Lambda_{n,\ell}}\norm{f(X_\alpha)-c_\alpha}_{\delta}^\delta\define\kA_\delta\in(0,\infty)$ for some $\delta\in(2,\infty)$. Let $U_n$ and $V_n$ be as in Corollary \ref{cor:U-stats-bounded}. 
	Define furthermore $\mathbf a_\delta\define \left( \kA_\delta / \abs{\Lambda_{n,\ell}}\right)^{1/\delta}=\left(\kA_\delta (n-\ell)!/n!\right)^{1/\delta}$. 
	 
	Then if $\delta\in(2,3)$ and $\V[U_n]\neq 0$, 
	\begin{equation}\label{eq:U-stats-2} 
	\dkol\left(\frac{U_n-\E[U_n]}{\sqrt{\V[U_n]}},\mathcal N(0,1)\right)\le 8.015  \left(\frac{\ell^2 (m+1)}n\right)^{\frac{\delta-2}{2(\delta+1)}} \Xi_\delta^{-\frac{\delta}{\delta+1}}, 
	\end{equation} 
	where 
	\begin{equation}\label{eq:U-stats-notation} 
	\Xi_\delta\define \frac{1}{\mathbf a_\delta} \sqrt{\frac{\V[V_n]}{\ell^2(m+1) n^{2l-1}}}\in(0,1] 
	\end{equation}	 
	If instead $\delta\in[3,\infty)$ and $\V[U_n]\neq 0$, then, using the same notation as in \eqref{eq:U-stats-notation}, 
	\begin{multline}\label{eq:U-stats-3} 
		\dkol\left(\frac{U_n-\E[U_n]}{\sqrt{\V[U_n]}},\mathcal N(0,1)\right)\\\le \max\left\{18.96 \ \left(\frac {\ell^2(m+1)}n\right)^{\frac{\delta-2}{2(\delta+1)}} \Xi_\delta^{-\frac{\delta}{\delta+1}}, 227.5 \ \sqrt{\frac{\ell^2(m+1)}n}\Xi_\delta^{-3} \right\}. 
	\end{multline} 
\end{corollary}

\begin{proof} 
	Recall that the r.v. $(f(X_\alpha))_{\alpha\in\Lambda_{n,\ell}}$ have a dependency graph with cardinality $N = \frac{n!}{(n-\ell)!}$ and maximal degree $D\le \frac{\ell^2(m+1)}{n} N -1 $. 
	Therefore, since $N\le n^\ell$,  
	\begin{equation}\label{eq:upper-bound-N(D+1)} 
		N(D+1)\le\frac{\ell^2 (m+1)}n N^2 \le \ell^2 (m+1) n^{2l-1}. 
	\end{equation} 
	If $\delta\in (2,3)$, Theorem \ref{thm:result-refined-delta<2} gives 
	\begin{equation*} 
		\dkol\left(\frac{U_n-\E[U_n]}{\sqrt{\V[U_n]}},\mathcal N(0,1)\right)\le 8.015\left( \xi_\delta \left(\frac{N}{D+1}\right)^{\frac 1 2 - \frac 1 \delta} \right)^{-\frac{\delta}{\delta+1}}, 
	\end{equation*} 
	where 
	\begin{equation*} 
		\xi_\delta = \left(\frac N{\mathcal A_\delta}\right)^{\frac 1\delta}\sqrt{\frac{\V[V_n]}{N(D+1)}}. 
	\end{equation*} 
	By \eqref{eq:upper-bound-N(D+1)}, $\xi_\delta\ge\Xi_\delta$. (In particular, $\Xi_\delta\le 1$ by Proposition \ref{prop:decrease-xi}.) 
	Furthermore, by \eqref{eq:bound-D-U}, 
	\begin{equation*} 
		\left(\frac{D+1}N\right)^{\frac{\delta-2}{2(\delta+1)}}\le \left(\frac{\ell^2(m+1)}{n}\right)^{\frac{\delta-2}{2(\delta+1)}}. 
	\end{equation*} 
	We thus obtain \eqref{eq:U-stats-2}. 
	If $\delta\in[3,\infty)$, a similar argument based on Theorem \ref{thm:dkol-delta-3+} gives \eqref{eq:U-stats-3}. 
\end{proof}

Let us compare Corollary \ref{cor:U-stats} with \cite[Theorem 3.14]{JansonU2021}, more specifically its improvement sketched in \cite[Remark 9.1]{JansonU2021}. 
The same observations as after Corollary \ref{cor:U-stats-bounded} still hold. 
The value of Corollary \ref{cor:U-stats} thus lies in the weaker assumption that $\mathcal{A}_\delta < \infty$, where we can take $\delta\in (2,\infty)$, whereas \cite[Theorem 3.14]{JansonU2021} required $\delta\geq 4$. More precisely, \cite[Remark 9.1]{JansonU2021} finds $\dkol((U_n-\E[U_n])/\sqrt{\V[U_n]}) = O(n^{-1/4})$ whenever $\mathcal{A}_4 < \infty$ and $O(n^{-1/2})$ when $\mathcal{A}_6 < \infty$. Our bound is thus better than Janson’s if and only if $\delta<4$ or $\delta \in [5,6)$. This is more easily seen in the following Corollary, which uses \cite[Theorem 3.8 and Theorem 8.1]{JansonU2021}.

\begin{corollary}\label{cor:U-stats-2} 
	Fix $m\in\N$. Let $(X_n)_{n\in\N}$ be a stationary, $m$-dependent sequence of random variables with values in a measurable space $\kS$ and let $f:\kS^\ell\to\R$ be measurable such that there exists $\delta\in(2,\infty)$ with $f(X_\alpha)\in L^\delta$ for every $\alpha\in\Lambda_{\infty, \ell}\define\bigcup_{n\in\N}\Lambda_{n,\ell}$. Recall $U_n$ defined by \eqref{eq:def-U-stat} and $V_n\define l! \binom nl U_n$. 
	 
	Then either $\V[V_n] = O(n^{2l-2})$, as $n\to\infty$ 
	or there exists $\mathcal K>0$ such that $\lim_{n\to\infty}\frac{\V[V_n]}{n^{2l-1}} = \mathcal K^2$. In the latter case, for $n$ large enough that $\V[U_n]\ge \frac{n^{2l-1}\kK^2}2$, the following is true. 
	 
	Recall $\mathbf a_\delta$ defined in Corollary \ref{cor:U-stats}. Then, with the notation from \eqref{eq:U-stats-notation}, if $\delta\in(2,3)$, 
	\begin{equation*} 
		\dkol\left(\frac{U_n-\E[U_n]}{\sqrt{\V[U_n]}},\kN(0,1)\right)\le 11.335 \ \mathcal K (\ell^2 (m+1))^{\frac{\delta-1}{\delta+1}}  \mathbf a_\delta^{\ \frac \delta{\delta+1}}  n^{-\frac{\delta-2}{2(\delta+1)}} . 
	\end{equation*} 
	If instead $\delta\in[3,\infty)$, then 
	\begin{multline*} 
		\dkol\left(\frac{U_n-\E[U_n]}{\sqrt{\V[U_n]}},\kN(0,1)\right) \\ 
		\le 
		\max\bigg\{ 26.672 \, (\ell^2 (m+1))^{\frac{\delta-1}{\delta+1}} \mathbf a_\delta^{\ \frac \delta{\delta+1}} n^{-\frac{\delta-2}{2(\delta+1)}}  \ , \ 643.5 \, (\ell^2 (m+1))^2 \mathbf a_\delta^{3} n^{-1/2}  \bigg\} . 
	\end{multline*} 
\end{corollary} 
\begin{proof} 
	By \cite[Theorem 3.8]{JansonU2021}, we have $\lim_{n\to\infty}\frac{\V[V_n]}{n^{2l-1}} = \kK^2$ for some $\kK\ge 0$. By \cite[Theorem 8.1]{JansonU2021}, we either have $\V[V_n] = O(n^{2l-2})$ or $\kK>0$. In the latter case, we can apply Corollary \ref{cor:U-stats} by noting that for $n$ large enough we have $\V[V_n]\ge\frac{\kK^2}{2} n^{2l-1}$, so that 
	\begin{equation*} 
		\Xi\ge \frac1{\mathbf a_\delta} \sqrt{\frac{1}{2 \ell^2 (m+1)}}. 
	\end{equation*} 
	Inserting this into Corollary \ref{cor:U-stats} gives the desired results. 
\end{proof}

\subsection{Application to estimating stock volatility}\label{sec:stock-volatility}

Let $(\mathscr{R}_t)_{t\geq 0}$ be a stochastic process and $(t_k)_{k\in \N}$ an unbounded, strictly increasing sequence of (deterministic) times with $t_0=0$. Define $\kappa_k = t_k - t_{k-1}$ and $X_k = \mathscr{R}_{t_k} - \mathscr{R}_{t_{k-1}}$ for every $k\geq 1$. We assume that $\E[X_k] = \mathfrak{e} \kappa_k$ and $\V[X_k] =\nu \kappa_k$. 
 
We can estimate $\mathfrak{e}$ and $\nu$ by weighted linear regression. The least square estimators, which are the maximum likelihood estimators assuming that the $(X_k)_{k\geq 1}$ are independent and normally distributed, are
\begin{align*} 
\hat{\mathfrak{e}}_n = \frac{1}{t_n} \sum_{k=1}^n X_k \qquad , \qquad \hat\nu_n = \frac{1}{n} \sum_{k=1}^n \frac{X_k^2}{\kappa_k} - \frac{t_n}{n} \hat{\mathfrak{e}}_n^2 \ .
\end{align*} 
$\hat\nu_n$ can be made unbiaised by replacing $n$ by $n-1$ in the denominator of the fractions. 

\begin{example}\label{ex:finance}
In the Black--Scholes model, a stock price is given by $\exp(\mathscr{R}_t)$ where $(\mathscr{R}_t)_{t\geq 0}$ is a Brownian motion with variance $\nu$ and drift $\mathfrak{e}$. The family $(X_k)_{k\in\N}$ then consists in returns computed between unevenly spaced epochs. 
In practice, $\hat{\mathfrak{e}}_n$ converges too slowly, so less direct methods, like the Capital Asset Pricing Model, are used. Furthermore, we are mainly interested in $\nu$, since $\mathfrak{e}$ plays no role when pricing options. 
\end{example} 

In a simple model, the $X_k$ are i.i.d., and we can establish a Law of Large Numbers and Central Limit Theorem for $\hat \nu_n$ under the condition that $X_k \in L^4$ together with some control on $\kappa_k$ and the moments of $X_k$. 
For more realism, the i.i.d. assumption needs to be relaxed. If we keep the independence property, a common approach (GARCH model, the RiskMetrics variance model) is to have the variance of the $(X_k)_{k\geq 1}$ evolve stochastically with time, with the $(X_k)_{k\geq 1}$ being independent conditionally on the variance process. If we instead relax the independence assumption, it is still possible to obtain a Strong Law of Large Numbers if one assumes that the $X_k$ are merely pairwise independent, using \cite[Corollary 2.1 and Remark 3]{Janisch2021}. 

We use our results to study the case where each $X_k$ is allowed to depend on a small number of other $X_j$. For simplicity, we make the unrealistic assumption that $\V[X_k] = \nu \kappa_k$, precisely the assumption that the RiskMetrics and GARCH models relax. This means that our estimations are only valid on short timescales where the returns can be assumed to be homoscedastic. A strong point of Corollary \ref{cor:main-CLT-theorem-finance} is that it gives an upper bound on the error of the least square estimator when the number of samples is limited.

\begin{corollary}\label{cor:main-CLT-theorem-finance}
	Let $n\in\N$, $\nu \in (0,\infty)$, $\delta > 4$, $(t_k)_{0\leq k \leq n}$ strictly increasing with $t_0=0$ and define $\kappa_k = t_k-t_{k-1}$ for every $1\leq k \leq n$. We consider a family $(X_k)_{1\le k \le n}$ with a dependency graph of maximal degree $m$ such that $\E[X_k] = 0$, $\V[X_k] = \nu \kappa_k$ and $X_k \in L^\delta$ for every $k$. 
	Write 
	\begin{equation*}
	\kT \define \frac{2^{\frac{\delta-1}{2}}}{n} \sum_{i=1}^n \left( \kappa_i^\delta + \frac{1}{2}\left(\frac{t_n}{n}\right)^\delta\right) \E\left[\abs{\frac{X_i}{\kappa_i}}^\delta\right]
	\end{equation*} 
	and assume that 
	\begin{equation*} 
		\V[\hat\nu_n]\ge\frac{\kK^2}n 
	\end{equation*} 
	for some $\kK>0$. 
	Then, if $\delta\in(4,6)$, we have for $\tilde\delta\define\frac\delta 2$, 
	\begin{equation}\label{eq:second-part-of-main-CLT-theorem} 
	\dkol\left(\frac{\hat\nu_n-\nu}{\sqrt{\V[\hat\nu_n]}}, \kN(0,1)\right) \le 8.015 \left(\frac{n}{\kK t_n}\right)^{\frac{\tilde\delta}{\tilde\delta+1}} \kT^{\frac{1}{\tilde\delta+1}} (4(m+1))^{\frac{\tilde\delta-1}{\tilde\delta+1}} n^{-\frac{\tilde\delta-2}{2(\tilde\delta+1)}}. 
	\end{equation} 
	If instead $\delta\in[6,\infty)$, then 
	\begin{multline*} 
		\dkol\left(\frac{\hat\nu_n-\nu}{\sqrt{\V[\hat\nu_n]}}, \kN(0,1)\right) \le \max\bigg\{18.96 \left(\frac{n}{\kK t_n}\right)^{\frac{\tilde\delta}{\tilde\delta+1}}  \kT^{\frac{1}{\tilde\delta+1}}  (4(m+1))^{\frac{\tilde\delta-1}{\tilde\delta+1}}  n^{-\frac{\tilde\delta-2}{2(\tilde\delta+1)}} ,\\ 
		227.5 \left(\frac{n}{\kK t_n}\right)^{3} \kT^{\frac{3}{\tilde\delta}} \frac{(4(m+1))^2}{\sqrt n}\bigg\}. 
	\end{multline*} 
	Assume $\limsup_{n\to\infty} \kT <\infty$ and $\limsup_{n\to\infty} n/t_n < \infty$. Then if $\delta\in(4,6)$ and $m+1=o(n^{\frac{\tilde\delta-2}{2\tilde\delta-2}})$, or $\delta\in[6,\infty[$ and $m+1=o(n^{\frac 14})$, 
	\begin{equation*} 
		\frac{\hat\nu_n-\nu}{\sqrt{\V[\hat\nu_n]}}\to\kN(0,1) 
	\end{equation*} 
	in distribution as $n\to\infty$. 
\end{corollary}

\begin{remark} 
	Note that $\V[\hat\nu_n]\ge\frac{\kK^2}n$ for some $\kK>0$ is true if $t_n = n$ and the $(X_k)_{k\in\N}$ are i.i.d. in $L^4$ with $\V[X_1]>0$. Indeed, in that case $\hat\nu_n$ is simply the usual variance estimator $\hat\Sigma_n$, and the representation $\hat\Sigma_n = \frac{1}{2\binom n2}\sum_{\set{i,j}\subset\binom{\llbracket 1,n\rrbracket}2}(X_i-X_j)^2$ gives 
	\begin{equation*} 
		\V[\hat\nu_n] = \V[\hat\Sigma_n] = \frac{\E[(X_1-\E[X_1])^4]}{n}-\frac{\V[X_1]^2 (n-3)}{n(n-1)}\ge \frac{\kK^2}n	 
	\end{equation*} 
	for some $\kK>0$ depending only on $X_1$. 
\end{remark} 
 
\begin{remark}
	Even if the $(X_k)_{k\in\N}$ were independent, we would need $X_k$ to be in $L^4$ to establish convergence of $\hat\nu_n$ to a normal distribution, and in order to obtain Berry--Esseen estimates on the speed of convergence one would need finite $\delta$-th moments for some $\delta>4$. Corollary \ref{cor:main-CLT-theorem-finance} matches these requirements. 
\end{remark}

\begin{proof}[Proof of Corollary \ref{cor:main-CLT-theorem-finance}] 
	Note that 
	\begin{align*}
	\hat\nu_n = \frac{1}{n t_n} \sum_{i,j=1}^n \frac{X_i}{\kappa_i} \left(\frac{t_n}{n} X_i - \kappa_i X_j \right) . 
	\end{align*}
	Write $Y_{i,j} = \frac{X_i}{\kappa_i} \left(\frac{t_n}{n} X_i - \kappa_i X_j \right)$. 
	From the Cauchy-Schwarz and AM-GM inequalities, for every $1\leq i,j\leq n$,
	\begin{align*} 
		\E\left[\abs{Y_{i,j}}^{\frac\delta2}\right]
		&\le \sqrt{\frac{1}{\kappa_i^{\delta}}\E[\abs{X_i}^\delta]\E\left[\abs{\frac{t_n}{n} X_i-\kappa_i X_j}^\delta\right]} \\
		&\le \sqrt{\frac{2^{\delta-1}}{\kappa_i^{\delta}}\E[\abs{X_i}^\delta] \E\left[\abs{\frac{t_n}{n} X_i}^\delta + \abs{\kappa_i X_j}^\delta\right]} \\
		&\le 2^{\frac{\delta-3}{2}} \left( \E[\abs{X_i}^\delta] + \left(\frac{t_n}{n \kappa_i}\right)^\delta\E[\abs{X_i}^\delta] + \E[\abs{X_j}^\delta] \right). 
	\end{align*} 
	We therefore have 
	\begin{equation*}
		\sum_{i,j=1}^n \E\left[ \abs{Y_{i,j}}^{\frac{\delta}{2}}\right] \leq n 2^{\frac{\delta-1}{2}} \sum_{i=1}^n \left( \kappa_i^\delta + \frac{1}{2}\left(\frac{t_n}{n}\right)^\delta\right) \E\left[\abs{\frac{X_i}{\kappa_i}}^\delta\right] = n^2 \kT .
	\end{equation*}
	 
	The graph on $\llbracket 1,n\rrbracket^2$, in which two pairs $(i,j)$ and $(k,\ell)$ are connected if and only if $\set{i, j}\cap\set{k,\ell}\neq\emptyset$ is a dependency graph for the random variables $(Y_{i,j})_{1\leq i,j\leq n}$. This graph has $N=n^2$ vertices and maximal degree $D\leq 4(m+1)n-1$. 
	 
	Assume first that $\delta\in(4,6)$. Then we can use Theorem \ref{thm:result-refined-delta<2} applied to the family $(Y_{i,j})_{1\leq i,j\leq n}$:
	\begin{equation}\label{eq:bound-for-mathfrak-V} 
	\dkol\left(\frac{\hat\nu_n-\nu}{\sqrt{\V[\hat\nu_n]}}, \kN(0,1)\right) \le 8.015 \xi_\delta^{-\frac{\tilde\delta}{\tilde\delta+1}} \left(\frac{D+1}N\right)^{\frac{\tilde\delta-2}{2(\tilde\delta+1)}}, 
	\end{equation} 
	where 
	\begin{equation} \label{eq:inequality-for-sigma-mathfrak-V}
		\xi_{\tilde\delta} = \kT^{-\frac 1{\tilde\delta}} \sqrt{\frac{\V\left[n t_n\hat\nu_n\right]}{N(D+1)}} 
		\geq \kT^{-\frac 1{\tilde\delta}} \frac{t_n}{2n} \sqrt{\frac{n\V[\hat\nu_n]}{m+1}} 
		\geq \kT^{-\frac 1{\tilde\delta}} \frac{\kK t_n}{2n \sqrt{m+1}} 
	\end{equation} 
	and we used the assumption $\V[\hat\nu_n]\ge\frac{\kK^2}{n}$ in the last inequality. 
	Combining \eqref{eq:bound-for-mathfrak-V} with \eqref{eq:inequality-for-sigma-mathfrak-V} gives the first part of the Theorem.
	If $\delta\in[6,\infty)$, we proceed instead with Theorem \ref{thm:dkol-delta-3+}. 
\end{proof}

\section{Proof of the main results} 
\label{sec:refined} 
 
\subsection{Overview of the proofs} 
\label{sec:overview} 
In this section, we present a proof of Theorems \ref{thm:mod-phi-corollary-new} to \ref{thm:result-refined-delta<2}. The two main methods to prove Berry--Esseen-type bounds for random variables are either Stein's method (see e.g. \cite{Penrose2003}), or a range of mutually similar analytical methods relying on bounding the Fourier transform of the sum. The latter approach gives the current benchmark for i.i.d. random variables, and it is the one we use here.  For a more detailed overview of related work, see section \ref{sec:previous-results}. 
 
The starting point of the proofs, established in Section \ref{sec:proof-cumulants}, is a bound of the cumulants of the sum $S$ of $N\in\N$ random variables $(Y_k)_{k\in V}$ indexed by a set $V$ with a dependency graph of maximum degree $D$. We establish this bound by an elementary refinement of a similar estimate from \cite{feray2013mod}. 
Recall that the cumulants $(\kappa^{(r)}(X))_{r\geq 1}$ of a bounded random variable $X$ are such that for every $z\in \C$, 
$$ \ln \E[\e^{z X}] = \sum_{r\geq 1} \frac{\kappa^{(r)}(X)}{r!} z^r . $$ 
In particular, $\kappa^{(1)}(X)=\E[X], \kappa^{(2)}(X) =\V[X]$ and $\kappa^{(3)}(X) = \E[(X-\E[X])^3]$; higher order cumulants can still be expressed in terms of the moments of $X$ but with more complex formulas. 
 
Assuming that the $(Y_k)_{k\in V}$ are uniformly bounded by the same constant $L$, our bound looks as follows: For every $r\geq \delta>1$, 
$$ \kappa^{(r)}(S) \leq (2(D+1))^{r-1}  r^{r-2} L^{r} \frac{\kA_\delta}{L^\delta} , $$ 
where $\kA_\delta$ is as in Definition \ref{def:renormedSD}. 
This bound is not optimal, but it is close enough to optimal for our purposes, as we justify with the following two observations. 
Firstly, although the dependency in $N$ and $D$ may be improved upon on a case-by-case basis, it is the best one can achieve in general. Indeed, if we let $(X_j)_{j\geq 0}$ be i.i.d. and $Y_{(D+1)i+1} = Y_{(D+1)i+2} = \dots = Y_{(D+1)(i+1)} = X_i$ for every $0 \leq i < N/(D+1)$, then 
$$ \kappa^{(r)}(S) = (D+1)^r \kappa^{(r)}\left(\sum_{i=0}^{\frac{N}{D+1}-1} X_i \right) = N (D+1)^{r-1} \kappa^{(r)}(X_1) . $$ 
Noting that $\kA_\delta / L^\delta \leq N$ gives the same dependency in $D$ and $N$ as our bound. 
Secondly, if $\|X_i\|_\infty=L$ then $\limsup_r (\kappa^{(r)}(X_1)/r!)^{1/r} = L$. In comparison, our bounds gives $2\e L$. The factor $2\e $ only worsens the final Berry--Esseen bounds by a constant factor; since we did not try to get the best possible constants it is of little impact. 
 
In Section \ref{sec:proof-fourier}, we translate our bound on the cumulants of $S$ to a Taylor expansion-type bound on the Fourier transform of $S$ in a neighborhood of zero, see Lemma \ref{lem:zone-of-control-bounded}: for every $s$ small enough and for some $C$ that depends on $N$, $D$, $\kA_3$, $\V[S]$ and $L$, 
\begin{equation}\label{eq:taylor-fourier} 
\left| \E\left[ \e^{\i s W} \right] - \e^{-\frac{s^2}{2}} \right| \leq C s^3 e^{C s^3 - \frac{s^2}{2}} , 
\end{equation} 
where we recall $W = (S-\E[S])/\sqrt{\V[S]}$. 
Fourier analysis is then used to derive a bound on $\dkol(W,\kN(0,1))$. 
We use the following elementary inequality that implies the classical Berry--Esseen bound (albeit with a suboptimal constant), cf. \cite[Chapter XVI, (3.13), (5.4)]{feller}: for every random variable $X$ with $\E[X]=0$ and $\V[X]=1$ and for every $T>0$, 
\begin{equation}\label{eq:dkol-feller} 
\dkol(X,\kN(0,1)) \leq \frac{1}{\pi} \int_{-T}^T \left| \E\left[\e^{\i s X}\right] - \e^{-\frac{s^2}{2}} \right| \frac{1}{|s|} \d s + \frac{24}{T\pi\sqrt{2\pi}} . 
\end{equation} 
One must then optimize over $T$. Indeed, the first term of the right-hand side of \eqref{eq:dkol-feller} increases with $T$, whereas the second term decreases with $T$. 
This is the content of Section \ref{sec:proof-bounded} and Section \ref{sec:proof-bounded-plus}, proving Theorems \ref{thm:mod-phi-corollary-new} and \ref{thm:mod-phi-corollary-new-exact} respectively. The proof of these two Theorems are nearly identical, only differing in the estimate on the Fourier transform of $S$. 
 
The case where the $(Y_k)_{k\in V}$ are not bounded requires some more work. The key idea is to re-use the estimates from the bounded case by ``truncating'' the $Y_k$. Fixing some $L>0$, and letting $Y^{(L)}_k = Y_k \ind{|Y_k|\leq L}$ (the actual formula is slightly different), we can apply our results to $S^{(L)} = \sum_k Y^{(L)}_k$. From there, a straightforward approach (used in \cite{feray2017mod}) is to bound the Kolmogorov distance between $S$ and $S^{(L)}$, and to optimize over $L$ to get the best possible bound. We choose a different approach: instead, we bound the difference between the Fourier transform of $S$ and that of $S^{(L)}$ in a neighborhood of zero, which gives an estimate reminiscent of \eqref{eq:taylor-fourier} (see Lemma \ref{lem:zone-of-control-unbounded}), and then proceed with the same proof method as Theorem \ref{thm:mod-phi-corollary-new}. The cases $\delta\geq 3$ and $\delta\in (2,3]$ are handled separately, in Sections \ref{sec:proof-3+} and \ref{sec:proof-3-} respectively. 
 
We tried to obtain the best constants that our proof methods could give, but only when it did not add too many technicalities. Better constants could be obtained by being more careful in the estimates that we use, or under additional hypotheses, or by using an alternative Fourier-transform-to-Kolmogorov-distance bound, like the one used in \cite{Shevtsova2013} (which gives the best Berry--Esseen bound for i.i.d. random variables as of 2012). A better bound on the cumulants of $S$ would also naturally translate to a better Berry--Esseen estimate.

\subsection{Cumulants of the sum} 
\label{sec:proof-cumulants} 
Recall the notations introduced in the introduction. 
Fix some real-valued family $(c_k)_{k\in V}$. For $L > 0$ and $k\in V$, define 
\begin{equation*} 
Y^{(L)}_k \define (Y_k-c_k)\ind{\abs{Y_k-c_k}\leq L} \ . 
\end{equation*} 
Recall 
\begin{equation*} 
S \define \sum_{k\in V} Y_k \qquad,\qquad v \define \sqrt{\V[S]} \ , 
\end{equation*} 
and define in a similar way 
\begin{equation*} 
S^{(L)} \define \sum_{k\in V} (Y^{(L)}_k+c_k) \qquad,\qquad v_L \define \sqrt{\V[S^{(L)}]} \ . 
\end{equation*} 
Recall also $\kA_\delta = \sum_{k\in V} \E[|Y_k - c_k|^\delta]$ from Definition \ref{def:renormedSD}, and $W = (S-\E[S])/v$. 
 
We start by proving a bound on the cumulants of $S^{(L)}$. 
Recall the definition of cumulants: for every real-valued random variable $Y$, if $\ln\E[\e^{cY}]$ is well-defined for $c\in\C$ in a neighborhood of zero, 
then it is analytic: 
\begin{equation*}\label{eq:def-cumulants} 
\ln \E[\e^{cY}] = \sum_{r\geq 1} \frac{\kappa^{(r)}(Y)}{r!} c^r 
\end{equation*} 
for some coefficients $\kappa^{(r)}(Y)$, called the \emph{$r$-th cumulants} of $Y$. In particular, $\kappa^{(1)}(Y) = \E[Y]$ and $\kappa^{(2)}(Y) = \V[Y]$. Furthermore, for every constant $c$ we have $\kappa^{(r)}(Y+c) = \kappa^{(r)}(Y)$ for every $r\neq 1$. 
 
\begin{lemma} \label{lem:better-cumulants} 
	For every $\delta\in[1,\infty)$, for every $r\in [\delta,\infty)$ with $r>1$ and for every $L>0$, 
	\begin{equation}\label{eq:improved-cumulant-bound} 
	\abs{\kappa^{(r)}(S^{(L)})} \leq r^{r-2} (2(D+1))^{r-1} L^{r} \frac{\kA_\delta}{L^\delta}  . 
	\end{equation} 
\end{lemma} 
 
This is an improvement of the bound \eqref{eq:cumulant-bound-1} in \cite[Theorem 9.8]{feray2013mod}, which fails to take advantage of the hypothesis $\delta > 1$, so that \cite{feray2013mod} obtained only the special case of Lemma \ref{lem:better-cumulants} where $\delta=1$: 
\begin{equation} 
\label{eq:cumulant-bound-1} 
\abs{\kappa^{(r)}(S^{(L)})} \leq \left( \sum_{k\in V} \|Y_k\|_1 \right) r^{r-2} (2(D+1))^{r-1} L^{r-1} \ . 
\end{equation}

\begin{proof} 
	First note that the cumulants of $S^{(L)}$ are well-defined because $|S^{(L)}-\sum_{k\in V} c_k|$ is almost surely bounded by $NL$. 
	Consider $r\geq \delta$. 
	Let $\pi$ be a partition of $\{1,2,\dots, r\}$ and $\alpha=(\alpha_1,\dots, \alpha_r)$ an element of $V^r$. Observe that every dependency graph of $(Y_\alpha)_{\alpha\in V}$ is also a dependency graph of $(Y^{(L)}_\alpha)_{\alpha \in V}$. 
	Consider 
	$$ \prod_{B \in \pi} \E\left[ \prod_{i\in B} Y^{(L)}_{\alpha_i} \right] . $$ 
	This number is denoted by $M_\pi$ in \cite[Section 9.4.1]{feray2013mod}, and we keep this notation. Using the generalized Hölder inequality, for every $B\in \pi$, 
	\begin{equation*} 
	\abs{\E\left[\prod_{i\in B}Y^{(L)}_{\alpha_i}\right]} = \abs{\E\left[\prod_{i\in B}Y^{(L)}_{\alpha_i} \prod_{i\notin B} 1 \right]} \le \left(\prod_{i\in B}\norm{Y^{(L)}_{\alpha_i}}_{r} \right)\left(\prod_{i\notin B}\norm{1}_{r} \right) = \prod_{i\in B}\norm{Y^{(L)}_{\alpha_i}}_{r} \ . 
	\end{equation*} 
	By using the arithmetico-geometric inequality, 
	\begin{equation*} 
	M_\pi \leq \prod_{i=1}^r \| Y_{\alpha_i}^{(L)} \|_r \leq \frac{1}{r} \sum_{i=1}^r \E\left[ \abs{Y_{\alpha_i}^{(L)}}^r \right] . 
	\end{equation*} 
	Replacing the bound used in \cite[Section 9.4.2]{feray2013mod} by the above, improved bound, and continuing with the same simplification as in their (47), 
	\begin{align}\label{eq:bound-with-ST-on-cumulants} 
	\left| \kappa^{(r)}\left( S^{(L)}-\sum_{k\in V} c_k \right)\right| &\leq 2^{r-1} \sum_{\alpha \in V^r} \frac{1}{r} \left( \sum_{i=1}^r \E\left[ \abs{Y_{\alpha_i}^{(L)}}^r \right] \right) \mathrm{ST}_{G[\alpha_1, \dots, \alpha_r]}, 
	\end{align} 
	where $\mathrm{ST}_{G[\alpha_1, \dots, \alpha_r]}$ denotes the number of spanning trees in the graph $G[\alpha_1, \dots, \alpha_r]$ (for the Definition of $G[\alpha_1,\dots,\alpha_r]$ see \cite[Section 9.3.2]{feray2013mod}). 
	The term $\mathrm{ST}_{G[\alpha_1, \dots, \alpha_r]}$ is invariant by permutation of the $(\alpha_i)_{1\leq i \leq r}$: therefore, for every $i\in\set{1,\dots, r}$, 
	\begin{equation}\label{eq:we-use-permutation-invariance-of-ST} 
	\sum_{(\alpha_1,\dots,\alpha_r)\in V^r}\E\left[\abs{Y_{\alpha_i}^{(L)}}^r\right] \mathrm{ST}_{G[\alpha_1,\dots,\alpha_r]}=\sum_{(\alpha_1,\dots,\alpha_r)\in V^r} \E\left[\abs{Y_{\alpha_1}^{(L)}}^r\right] \mathrm{ST}_{G[\alpha_1,\dots,\alpha_r]}. 
	\end{equation} 
	Furthermore, for any fixed $\alpha_1\in V$, the sum of $\mathrm{ST}_{G[\alpha_1, \dots, \alpha_r]}$ over $(\alpha_2, \ldots, \alpha_r) \in V^{r-1}$ is bounded from above by $r^{r-2} (D+1)^{r-1}$ by \cite[Corollary 9.17]{feray2013mod}. Combining this with \eqref{eq:bound-with-ST-on-cumulants} and \eqref{eq:we-use-permutation-invariance-of-ST}, we get 
	\begin{align*} 
	\left| \kappa^{(r)}\left( S^{(L)}-\sum_{k\in V} c_k \right)\right| 
	&\leq r^{r-2} (2(D+1))^{r-1}  \sum_{k \in V} \E\left[ \abs{Y_k^{(L)}}^r \right] \ . 
	\end{align*} 
	Since $|Y^{(L)}_k|\leq L$, $\E\left[ \abs{Y_k^{(L)}}^r \right] \leq L^{r-\delta} \E[|Y^{(L)}_k|^\delta] \leq L^{r-\delta} \E[|Y_k-c_k|^\delta]$. 
	The Lemma then follows by the above observation that cumulants of order $r\geq 2$ are not changed by additive constants. 
\end{proof}

Next, we gather some useful estimates on $|\E[S^{(L)}]-\E[S]|$ and $|\V[S^{(L)}]-\V[S]|$. 
By Hölder's and Markov's inequality, defining $Z^{(L)}_k = (Y_k-c_k) \ind{|Y_k-c_k|>L}$: 
\begin{align*} 
\E[|Z^{(L)}_k|] \leq \|Y_k-c_k\|_\delta \P(|Y_k-c_k|>L)^{1-1/\delta} \leq \|Y_k-c_k\|_\delta^\delta L^{1-\delta} . 
\end{align*} 
Thus 
\begin{equation}\label{eq:bound-E(S) minus E(S^-)} 
\E[\abs{S-S^{(L)}}] \leq \sum_{k\in V} \E[\abs{ Y_k - Y^{(L)}_k - c_k}] = \sum_{k\in V} \E[|Z^{(L)}_k|] \leq  L^{1-\delta} \kA_\delta . 
\end{equation} 
Next, since $Y_i$ and $Y_j$ are independent (and thus uncorrelated) whenever there is no edge between $i$ and $j$ in the dependency graph, and since $Y_k = Y^{(L)}_k + Z^{(L)}_k + c_k$, 
\begin{align*} 
\lvert\V[S^{(L)}] - \V[S]\rvert &\leq 2 \sum_{i\sim j} \lvert\Cov(Y^{(L)}_i, Z^{(L)}_j)\rvert + \sum_{i\sim j} \lvert\Cov(Z^{(L)}_i, Z^{(L)}_j)\rvert 
\end{align*} 
where the sum is over $i,j\in V$ with $i\sim j$, and we write $i\sim j$ when $i$ and $j$ are adjacent in the dependency graph of $(Y_k)_{k\in V}$ or when $i=j$. For every $i,j\in V$, 
\begin{align*} 
\lvert\Cov(Y^{(L)}_i, Z^{(L)}_j)\rvert &\leq \sqrt{\V[Y^{(L)}_i \ind{|Y_j-c_j|>L}] \V[Z^{(L)}_j]} \\ 
&\leq \sqrt{\E[\left(Y^{(L)}_i\right)^2 \ind{|Y_j-c_j|>L}] \E[\left(Z^{(L)}_j\right)^2]} \\ 
&\leq \frac{1}{2}\left( \E\left[\left(Y^{(L)}_i\right)^2 \ind{|Y_j-c_j|>L}\right] + \E\left[\left(Z^{(L)}_j\right)^2\right] \right) \\ 
&\leq  \E[|Y_j-c_j|^\delta] L^{2-\delta} 
\end{align*} 
using the arithmetico-geometric inequality, the fact that $\E[(Z^{(L)}_j)^2] \leq \E[|Y_j-c_j|^\delta] L^{2-\delta}$ by the same argument used above to bound $\E[|Z^{(L)}_j|]$, and 
$$\E\left[\left(Y^{(L)}_i\right)^2 \ind{|Y_j-c_j|>L}\right] \leq L^2 \P(|Y_j-c_j|>L) \leq \E[|Y_j-c_j|^\delta] L^{2-\delta}. $$ 
Similarly, 
\begin{multline*} 
\lvert\Cov(Z^{(L)}_i, Z^{(L)}_j) \rvert \leq \sqrt{\V[Z^{(L)}_i] \V[Z^{(L)}_j]} \\ 
\leq \frac{1}{2}\left( \V[Z^{(L)}_i] + \V[Z^{(L)}_j]\right) \leq \frac{1}{2} \left( \E[|Y_i-c_i|^\delta] + \E[|Y_j-c_j|^\delta] \right) L^{2-\delta} . 
\end{multline*} 
We finally get 
\begin{equation}\label{eq:bound difference variances} 
\lvert\V[S^{(L)}] - \V[S]\rvert \leq \sum_{k\in V} 3 L^{2-\delta} (\mathrm{deg}(k)+1) \E[|Y_k-c_k|^\delta] \leq 3 L^{2-\delta} (D+1) \kA_\delta 
\end{equation} 
where $\mathrm{deg}(k)$ is the degree of $k$ in the dependency graph of $(Y_k)_{k\in V}$.

\subsection{Fourier transform of the sum}\label{sec:proof-fourier}

The core estimate is a bound that allows us to compare the Fourier transform of the sum $S$ to that of a normally distributed random variable with same expectation and variance. Lemma \ref{lem:zone-of-control-bounded} contains such a bound for the sum $S^{(L)}$ of the random variables $(Y^{(L)}_k)_{k\in V}$, which are obtained by ``truncating'' the $(Y_k)_{k\in V}$, see the beginning of Section \ref{sec:proof-cumulants}. 
 
\begin{lemma} 
\label{lem:zone-of-control-bounded} 
Write $\delta' = \min(\delta,3)$. Let $L>0$ such that $v_L > 0$. 
For every $\xi \in \C$ such that $|\xi| \leq v_L / (2\e L (D+1))$, writing $K_L = C (D+1)^2 (L/v_L)^3 \kA_{\delta'}/L^{\delta'}$, 
\begin{equation*} 
\left|\E\left[\e^{\i\xi \frac{S^{(L)}-\E[S^{(L)}]}{v_L}}\right] \e^{\frac{\xi^2}{2}} - 1 \right| \leq 
K_L |\xi|^3 \e^{K_L |\xi|^3} , 
\end{equation*} 
with 
\begin{equation}\label{eq:Cseries-definition-b} 
C \define 4\e^3 \sum_{r\geq 3} \frac{r^{r-2}}{r! \e^r} \in [5.17, 5.18]. 
\end{equation} 
\end{lemma} 
 
This implies that $(S^{(L)}-\E[S^{(L)}])/(N^{1/3} (D+1)^{2/3} v_L)$ has a ``zone of control'' of index $(3,3)$ with respect to the reference law $\kN(0,1)$ (i.e. $\alpha=2, c=1/{\sqrt 2}$) 
in the sense of \cite[Definition 5]{feray2017mod}. 
 
\begin{proof}[Proof of Lemma \ref{lem:zone-of-control-bounded}] 
We have 
\begin{align*} 
\E\left[\e^{\i\xi\frac{S^{(L)}-\E[S^{(L)}]}{v_L}}\right]\exp\left(\frac{\xi^2}2\right) 
=\exp\left(\sum_{r=3}^\infty\frac{\kappa^{(r)}(S^{(L)})}{r!} \left(\i \frac\xi {v_L} \right)^r \right) 
\define\exp(z) . 
\end{align*} 
By the improved cumulant bound \eqref{eq:improved-cumulant-bound} with $\delta$ replaced by $\delta'$, 
\begin{align*} 
\lvert z\rvert &\leq \sum_{r=3}^\infty r^{r-2}(2 (D+1))^{r-1} L^r \frac{\kA_{\delta'}}{L^{\delta'}} \frac{\abs\xi^r}{r! v_L^r} = Z\left(2 \e (D+1) L \frac{\abs\xi}{v_L}\right), 
\end{align*} 
where 
$$Z(s) = \frac{1}{2(D+1)} \frac{\kA_{\delta'}}{L^{\delta'}} \sum_{r=3}^\infty \frac{r^{r-2}}{r! \e^r} s^r $$ 
is such that for every $s \in [0,1]$, 
$$ Z(s) \leq \left( \sum_{r\geq 3} \frac{r^{r-2}}{r!\e^r} \right) \frac{1}{2(D+1)}\frac{\kA_{\delta'}}{L^{\delta'}} s^3 . $$ 
The claim follows from $\abs{\e^z-1}\le\abs z\e^{\abs z}$. 
\end{proof}

\begin{remark}\label{rem:zone-of-control-refined} 
	By refining the proof of Lemma \ref{lem:zone-of-control-bounded}, it is possible to obtain the following estimate. Let $L>0$ such that $v_L>0$, and write $\delta" = \min(\delta,4)$ and $\rho_L = |\kappa^{(3)}(S^{(L)})|$. 
	Then for every $\xi\in\C$ such that $|\xi| \leq v_L / (2\e L (D+1))$, 
	\begin{equation*} 
		\left|\E\left[\e^{\i\xi \frac{S^{(L)}-\E[S^{(L)}]}{v_L}}\right] \e^{\frac{\xi^2}{2}} - 1 \right| \leq \xi^2 x(\xi) \e^{\xi^2 x(\xi)} , 
	\end{equation*} 
	where 
	\begin{equation*} 
		x(\xi) \define \frac{\rho_L}{6v_L^3} |\xi| + C" (D+1)^3 \left(\frac{L}{v_L}\right)^4 \frac{\kA_{\delta"}}{L^{\delta"}} |\xi|^2 
	\end{equation*} 
	and 
	$$ C" \define 8\e^4 \sum_{r\geq 4} \frac{r^{r-2}}{r! \e^r} = 2\e(C-2) \in [16.8, 17.3] 
	. $$ 
	Assuming $\delta > 3$, 
	with similar computations as those we used to establish \eqref{eq:bound difference variances} we can show that 
	$$ |\kappa^{(3)}(S^{(L)}) - \kappa^{(3)}(S)| \leq 21 (D+1)^2 L^3 \frac{\kA_{\delta}}{L^{\delta}} . $$ 
\end{remark}

Our next step is to control the Fourier transform of $S$ the sum of the non-truncated random variables. The following Lemma gives a family of bounds on the Fourier transform of $S$, indexed by the truncation $L$. Compared to Lemma \ref{lem:zone-of-control-bounded}, the formula contains additional terms to account for the tail of the distributions of the $(Y_k)_k$. 
 
\begin{lemma}\label{lem:zone-of-control-unbounded} 
Assume $v\neq 0$. For every $s\in\R\setminus\set 0$, define 
\begin{equation}\label{eq:def-u-J} 
w = w(s) = \frac{L(D+1)|s|}{v} 
\end{equation} 
Recall $C\in[5.17,5.18]$ defined in Lemma \ref{lem:zone-of-control-bounded}. Let $\delta' = \min(3,\delta)$. 
Then for every $0 < L \leq v/(2\e |s| (D+1))$ such that 
\begin{equation}\label{eq:cond-u} 
v^2 > 3(D+1) L^2 \frac{\kA_{\delta}}{L^{\delta}} , 
\end{equation} 
we have 
\begin{multline*} 
(D+1) \left| \E\left[ \e^{\i s \frac{S-\E[S]}{v}} \right] - \e^{-\frac{s^2}{2}} \right| 
\ \leq 
\left(2w + \frac{3 w^2}{2}\right) \frac{\kA_{\delta}}{L^{\delta}} \\ 
+ C \frac{\kA_{\delta'}}{L^{\delta'}} w^3 \ \exp\left( -\frac{w^2}{2(D+1)^2} \left(\frac{v^2}{L^2} - 3 (D+1) \frac{\kA_{\delta}}{L^{\delta}} - 2C(D+1) \frac{\kA_{\delta'}}{L^{\delta'}} w \right) \right) . 
\end{multline*} 
\end{lemma}

\begin{proof} 
For every $s \in \R\setminus\{0\}$ and every $L>0$, 
\begin{multline*} 
\left| \E\left[\e^{\i s \frac{S-\E[S]}{v}}\right] - \e^{-\frac{s^2}{2}} \right| 
\leq \underbrace{\left| \E\left[\e^{\i s \frac{S-\E[S]}{v}}\right] - \E\left[\e^{\i s \frac{S^{(L)}-\E[S]}{v}}\right] \right|}_{\define\mathrm {I}} \\ 
+ \underbrace{\left| \E\left[\e^{\i s \frac{S^{(L)}-\E[S]}{v}}\right] - \e^{\i s \frac{\E[S^{(L)}]-\E[S]}{v}}\e^{- \frac 1 2 \left(\frac{v_L s}{v}\right)^2 } \right|}_{\define\mathrm{II}} 
+ \underbrace{\left| \e^{\i s \frac{\E[S^{(L)}]-\E[S]}{v}}\e^{- \frac 1 2 \left(\frac{v_L s}{v}\right)^2 } - \e^{-\frac{s^2}{2}}  \right|}_{\define\mathrm{III}}  . 
\end{multline*} 
 
Write $R^{(L)} \define S-S^{(L)} = \sum_{k\in V} Z^{(L)}_k$. 
Since the $(Y_k)_{k\in V}$ are real-valued we have \\$\abs{\exp(\i s \frac{S^{(L)}-\E[S]}v)}=1$. Using that $\abs{\e^{\i x}-1}\le\abs{x}$ for $x\in\R$, we bound I by 
\begin{align}\label{eq:bound-on-I} 
\mathrm{I} = \left| \E\left[ \e^{\i s \frac{R^{(L)}}{v}} - 1 \right] \right| 
\leq \frac{|s|}{v} \E[|R^{(L)}|] 
 \leq \frac{|s|L}{v} \frac{\kA_{\delta}}{L^{\delta}} = \frac{w}{D+1} \frac{\kA_{\delta}}{L^{\delta}} , 
\end{align} 
where the last inequality follows from \eqref{eq:bound-E(S) minus E(S^-)}. 
 
By assumption \eqref{eq:cond-u}, we have that $v_L > 0$ (since by \eqref{eq:bound difference variances}, $|v_L^2 - v^2| \leq  3(D+1) L^2 \kA_{\delta}L^{-\delta} < v^2$). 
We can use Lemma \ref{lem:zone-of-control-bounded}: for every $s$ such that $v_L|s|/v \leq v_L/(2\e L(D+1))$, 
\begin{equation}\label{eq:complicated-inequalities}\begin{split} 
\e^{\frac 1 2 \left(\frac{v_L s}{v}\right)^2 } \mathrm{II} &= \left| \e^{\frac 1 2 \left(\frac{v_L s}{v}\right)^2 } \E\left[\e^{\i s \frac{S^{(L)}-\E[S]}{v}}\right] - \e^{\i s \frac{\E[S^{(L)}]-\E[S]}v} \right| \\ 
&= \left| \e^{\frac 1 2 \left(\frac{v_L s}{v}\right)^2 } \E\left[\e^{\i \left(\frac{v_L s}{v}\right) \frac{S^{(L)}-\E[S^{(L)}]}{v_L}}\right] - 1 \right| \\ 
&\leq K_L \left(\frac{v_L |s|}{v}\right)^3 \e^{K_L\left(\frac{v_L |s|}{v}\right)^3} \\ 
&= K'_L |s|^3 \e^{K'_L|s|^3} \\ 
&= \frac{C}{D+1} \frac{\kA_{\delta'}}{L^{\delta'}} w^3 \e^{\frac{C}{D+1} \frac{\kA_{\delta'}}{L^{\delta'}} w^3} , 
\end{split}\end{equation} 
where $K_L$ is given by Lemma \ref{lem:zone-of-control-bounded} and 
\begin{equation*} 
K'_L \define C (D+1)^2 \frac{\kA_{\delta'}}{L^{\delta'}} \left(\frac L v\right)^3 \ . 
\end{equation*} 
For the last term, since $z\mapsto\e^z$ is 1-Lipschitz on $\{ z \in \C : \operatorname{Re}(z)\leq 0 \}$: 
\begin{align*} 
\mathrm{III} &\leq \abs{\i s \frac{\E[S^{(L)}]-\E[S]}{v} - \frac 1 2 \left(\frac{v_L s}{v}\right)^2 +\frac{s^2}{2}}\le \frac{|s|}{v} |\E[S^{(L)}]-\E[S]| + \frac{s^2}{2} \frac{|v_L^2 - v^2|}{v^2} . 
\end{align*} 
By \eqref{eq:bound-E(S) minus E(S^-)} and \eqref{eq:bound difference variances}, 
\begin{equation*} 
\begin{cases} 
\left|\E[S^{(L)}]-\E[S]\right| \leq \E[|S^{(L)}-S|] \leq L \frac{\kA_{\delta}}{L^{\delta}} , \\ 
|v_L^2 - v^2| \leq 3 (D+1) L^2 \frac{\kA_{\delta}}{L^{\delta}} . 
\end{cases} 
\end{equation*} 
We deduce 
\begin{equation}\label{eq:bound-on-III} 
\mathrm{III} \leq \frac{|s| L}{v} \frac{\kA_{\delta}}{L^{\delta}} + \frac{3 (D+1) s^2 L^2}{2 v^2} \frac{\kA_{\delta}}{L^{\delta}} = \frac{1}{D+1} \frac{\kA_{\delta}}{L^{\delta}} \left( w + \frac{3 w^2}{2}\right) \ . 
\end{equation} 
Combining \eqref{eq:bound-on-I}, \eqref{eq:complicated-inequalities} and \eqref{eq:bound-on-III} yields the desired result. 
\end{proof}

\subsection{Proof for bounded random variables} 
\label{sec:proof-bounded}

In this section, we prove Theorem \ref{thm:mod-phi-corollary-new}, which contains a Berry--Esseen-type bound on $S$ when the $(Y_k)_{k\in V}$ are a.s. uniformly bounded by some common constant. 
The proof relies on one hand on Lemma \ref{lem:zone-of-control-bounded}, and on the other hand on a Fourier-transform-to-Kolmogorov-distance inequality. Various options exist for this second step, differing by the constant they yield. We chose \eqref{eq:dkol-feller}, which allows for a short and relatively elementary proof. To improve the constant, one could refine our estimate for 
$$ \int_{-T}^T \left| \E\left[\e^{\i\xi W}\right] - \e^{-\frac{\xi^2}{2}} \right| \frac{\d\xi}{|\xi|} $$ 
where we recall $W = (S-\E[S])/v$, 
for example by a better bound on $|\e^z - 1|$ and a change of variable $u^2/2 = \xi^2/2 + z$ (where $z$ is given in the proof of Lemma \ref{lem:zone-of-control-bounded}). Using another Fourier-transform-to-Kolmogorov-distance inequality would also improve upon the constant, for example \cite{Shevtsova2013} which gives the state-of-the-art in the i.i.d. case as of 2012.

Let us now proceed with the proof of Theorem \ref{thm:mod-phi-corollary-new}. 
Recall that we consider random variables $Y_k$ such that 
\begin{equation*} 
\max_{k\in V} \|Y_k-c_k\|_\infty \leq L . 
\end{equation*} 
Then $S^{(L)} = S$ almost surely and in particular $v=v_L$. Let $G\sim \kN(0,1)$. By Lemma \ref{lem:zone-of-control-bounded}, for every $T \leq v/(2\e L(D+1))$ such that $2KT < 1$, 
\begin{align*} 
\int_{-T}^T \left| \E\left[\e^{\i\xi W}\right] - \E\left[\e^{\i\xi G}\right] \right| \frac{\d\xi}{|\xi|} 
&\leq K \int_\R |\xi|^2 \e^{-\frac{\xi^2}{2}(1-2KT)} \d\xi = \frac{K \sqrt{2\pi}}{(1-2KT)^{3/2}} 
\end{align*} 
where $K = K_L = C(D+1)^2 \kA_{\delta'} L^{3-\delta'}/v^3$ is given in Lemma \ref{lem:zone-of-control-bounded}. 
$K$ is minimal when taking $\delta=3$. Indeed, 
$$ \kA_{\delta'} L^{3-\delta'} = \sum_{k\in V} \E[|Y_k-c_k|^{\delta'}] L^{3-\delta'} \geq \sum_{k\in V} \E[|Y_k-c_k|^3] = \kA_3 . $$ 
 Using \eqref{eq:dkol-feller}, 
\begin{align}\label{eq:dkol_inf} 
\dkol(W,G) \leq \frac{1}{\pi} \frac{K\sqrt{2\pi}}{(1-2KT)^{3/2}} + \frac{24}{\pi \sqrt{2\pi} T} = K\sqrt{\frac{2}{\pi}} \left( (1-\alpha)^{-3/2} + \frac{24}{\pi\alpha} \right) \ , 
\end{align} 
where we write $\alpha = 2KT \leq Kv/(\e L(D+1))$. The quantity in brackets is convex on $(0,1)$ and has a unique minimum $I \leq 16.5653$ reached at $\alpha_0 \leq 0.636647$. For every $\alpha \leq \alpha_0$ the term in brackets is bounded from above by $\alpha_0 I / \alpha$, so that in fact 
$$ \inf_{\alpha<1 \ , \ \alpha \leq \frac{Kv}{\e L(D+1)}} \left[ (1-\alpha)^{-3/2} + \frac{24}{\pi\alpha} \right]  \leq  I \max\left\{1 , \alpha_0 \frac{\e L(D+1)}{Kv} \right\}  \ . $$ 
Using this in \eqref{eq:dkol_inf} yields 
\begin{align*} 
\dkol(W,G) 
&\leq \max\left\{ K I \sqrt{\frac{2}{\pi}} ,  I \sqrt{\frac{2}{\pi}} \alpha_0 \frac{\e L(D+1)}{v} \right\} \ . 
\end{align*} 
Replacing $K$ by its expression (with $\delta=3$ to have the best bound) and computing the constants finishes the proof of the Theorem.

\subsection{Improvement with third cumulant, bounded case}\label{sec:proof-bounded-plus}

We now prove Theorem \ref{thm:mod-phi-corollary-new-exact}. 
Since we closely follow the proof of Theorem \ref{thm:mod-phi-corollary-new}, we only highlight the differences. Let $G\sim \kN(0,1)$, $W = (S-\E[S])/v$ and recall $x(\xi)$ defined in Remark \ref{rem:zone-of-control-refined} as well as $\rho\define\abs{\kappa^{(3)}(S)}$. Then, by using Remark \ref{rem:zone-of-control-refined} instead of Lemma \ref{lem:zone-of-control-bounded} and writing $J = 4C"(D+1)^3 \kA_{\delta"} L^{4-\delta"} v^{-4}$, 
for every $T \leq v/(2\e L(D+1))$ such that $2x(T) < 1$, 
\begin{align*} 
\int_{-T}^T \left| \E\left[\e^{\i\xi W}\right] - \E\left[\e^{\i\xi G}\right] \right| \frac{\d\xi}{|\xi|} 
&\leq \frac{\rho}{6 v^3} \int_\R |\xi|^2 \e^{-\frac{\xi^2}{2}(1-2x(T))} \d\xi + \frac{J}{4} \int_\R |\xi|^3 \e^{-\frac{\xi^2}{2}(1-2x(T))} \d\xi \\ 
&\leq \frac{\sqrt{2\pi}}{6} \frac{\rho}{v^3} (1-2x(T))^{-3/2} + J(1-2x(T))^{-2} \ . 
\end{align*} 
Use this into \eqref{eq:dkol-feller}: 
\begin{align} 
\dkol(W,G) \leq \frac{\sqrt{2\pi}}{6\pi} \frac{\rho}{v^3} (1-2x(T))^{-3/2} + \frac{J}{\pi}(1-2x(T))^{-2}  + \frac{24}{\pi \sqrt{2\pi} T} \ . 
\end{align} 
Finding the $T$ for which the bound is smaller becomes more difficult; we instead content ourselves with taking $T = T_0 \define v/(2\e L(D+1))$ if $2x(T_0) < \alpha_0$ where $\alpha_0$ is the same as in the proof of Theorem \ref{thm:mod-phi-corollary-new}, and $T$ such that $2x(T) = \alpha_0$ otherwise. Then 
\begin{align}\label{eq:dkol_inf_b} 
\dkol(W,G) \leq \frac{\sqrt{2\pi}}{6\pi (1-\alpha_0)^{3/2}} \frac{\rho}{v^3} + \frac{J}{\pi(1-\alpha_0)^2}  + \frac{24}{\pi \sqrt{2\pi} T} \ . 
\end{align} 
The $T_1$ such that $2x(T_1)=\alpha_0$ is given by 
\begin{align*} 
T_1 = \frac{2}{J} \frac{\rho}{6 v^3} \left( - 1 + \sqrt{ 1 + \alpha_0 \frac{J}{2} \left(\frac{6v^3}{\rho}\right)^2 } \right). 
\end{align*} 
Using that for every $y>0$, 
$$ \frac{1}{\sqrt{1+y}-1} = \frac{1+\sqrt{1+y}}{y} \leq \frac{1}{\sqrt y} + \frac{2}{y} \ , $$ 
we obtain 
\begin{align*} 
\frac{1}{T_1} &\leq \frac{J}{2} \frac{6 v^3}{\rho} \left( \frac{\rho}{6v^3} \sqrt{ \frac{2}{\alpha_0 J}} + \frac{4}{\alpha_0 J} \left(\frac{\rho}{6v^3}\right)^2\right) \\ 
&\leq \sqrt{ \frac{J}{2 \alpha_0}} + \frac{2}{\alpha_0} \frac{\rho}{6v^3} \ . 
\end{align*} 
We have $2x(T_0)<\alpha_0$ if and only if $T_0 < T_1$, in which case $1/T_0 > 1/T_1$. Combined with \eqref{eq:dkol_inf_b} we get 
\begin{align}\label{eq:precis-final} 
\dkol(W,G) \leq& \frac{\sqrt{2\pi}}{6\pi (1-\alpha_0)^{3/2}} \frac{\rho}{v^3} + \frac{J}{\pi(1-\alpha_0)^2} + \frac{24}{\pi \sqrt{2\pi}} \max \left\{ \frac{2\e L(D+1)}{v} \ , \ \sqrt{ \frac{J}{2 \alpha_0}} + \frac{2}{\alpha_0} \frac{\rho}{6v^3} \right\} \\ 
\leq& 0.607148 \frac{\rho}{v^3} + 116.84 \frac{(D+1)^3 \kA_{\delta"} L^{4-\delta"}}{v^4}\nonumber\\ 
&+ \max \left\{ 16.57\frac{L(D+1)}{v} \ , \ 22.47 \sqrt{\frac{(D+1)^3 \kA_{\delta"} L^{4-\delta"}}{v^4}} + 1.596 \frac{\rho}{v^3} \right\} \ .\nonumber 
\end{align} 
By the same reasoning as in the proof of Theorem \ref{thm:mod-phi-corollary-new}, the smallest value for $J$ is attained for $\delta=4$, an assumption that we make to obtain the bound of Theorem \ref{thm:mod-phi-corollary-new-exact}.

To conclude, we prove that the bound of Theorem \ref{thm:mod-phi-corollary-new-exact} is always better than the bound of Theorem \ref{thm:mod-phi-corollary-new}, up to a constant multiplicative factor. Recall the latter: 
	\begin{align}\label{eq:precis-comparaison} 
		\dkol\left(W, \kN(0,1)\right) \leq \max\left\{ C I \sqrt{\frac{2}{\pi}} (D+1)^2 \frac{\kA_3}{v^3} ,  I \sqrt{\frac{2}{\pi}} \alpha_0 \frac{\e L(D+1)}{v} \right\} 
	\end{align} 
with, recalling $\alpha_0$ as in the proof of Theorem \ref{thm:mod-phi-corollary-new}, 
$$ I = \frac{1}{(1-\alpha_0)^{3/2}} + \frac{24}{\pi \alpha_0} . $$ 
Let us bound every term of \eqref{eq:precis-final}. First, we use that by \eqref{eq:improved-cumulant-bound}, 
$$ \rho \leq 12 (D+1)^2 \kA_3 \qquad \implies \qquad \frac{\sqrt{2\pi}}{6\pi} \frac{\rho}{v^3} \leq 2\sqrt{\frac{2}{\pi}} (D+1)^2 \frac{\kA_3}{v^3} \leq \frac{2}{CI} (\star) $$ 
where $(\star)$ is the right-hand side of \eqref{eq:precis-comparaison}. 
The second term of \eqref{eq:precis-final} can be written as 
$$ \frac{8\e (C-2)}{\pi (1-\alpha_0)^2} (D+1)^3 \frac{\kA_4}{v^4} \ = \ \frac{8\e (C-2)}{\pi (1-\alpha_0)^2} \frac{\kA_4}{L \kA_3} (D+1)^2 \frac{\kA_3}{v^3} \frac{L(D+1)}{v} \ . $$ 
Since $\kA_4 \leq L\kA_3$ and since the square of a maximum of two positive terms is larger than the product of the two terms, we deduce that the second term of \eqref{eq:precis-final} is smaller than 
\begin{multline*} \frac{8\e (C-2)}{\pi (1-\alpha_0)^2} \left( CI\sqrt{\frac{2}{\pi}} \cdot I\e\alpha_0\sqrt{\frac{2}{\pi}} \right)^{-1} (\star)^2 = \frac{8\e (C-2)}{\pi (1-\alpha_0)^2} \frac{\pi}{2 C I^2 \e \alpha_0} (\star)^2 = \frac{4(C-2)}{(1-\alpha_0)^2 C I^2 \alpha_0} (\star)^2 \ . 
\end{multline*} 
The first term in the maximum of \eqref{eq:precis-final} is smaller than 
$$ \frac{48 \e}{\pi \sqrt{2\pi}} \left( I\e\alpha_0 \sqrt{\frac{2}{\pi}} \right)^{-1} (\star) = \frac{24}{\pi I \alpha_0} \ (\star) \ . $$ 
To finish, we use the same bounds as before to get 
\begin{align*} 
\frac{24}{\pi\sqrt{2\pi}} \frac{2}{\alpha_0} \frac{\rho}{6v^3} &\leq \frac{24}{\pi \alpha_0} \left( 2\sqrt{\frac{2}{\pi}} (D+1)^2 \frac{\kA_3}{v^3} \right) \leq \frac{2}{CI}\frac{24}{\pi\alpha_0} \\ 
\frac{24}{\pi\sqrt{2\pi}} \sqrt{\frac{J}{2\alpha_0}} &\leq \frac{24}{\pi\sqrt{2\pi}} \sqrt{\frac{4\e(C-2)}{\alpha_0} \left( CI^2 \frac{2}{\pi} \e\alpha_0 \right)^{-1} } \ (\star) \\ 
&=  \frac{24}{\pi \alpha_0 I} \sqrt{\frac{C-2}{C} } \ (\star) 
\end{align*} 
Assuming $(\star)\leq 1$, the maximum in \eqref{eq:precis-final} is smaller than 
\begin{align*} 
\frac{24}{\pi I \alpha_0} \max\left\{ 1 \ , \ \frac{2}{C} + \sqrt{\frac{C-2}{C}} \right\} = \frac{24}{\pi I \alpha_0} \left( \frac{2}{C} + \sqrt{\frac{C-2}{C}} \right) \leq 0.8483 . 
\end{align*} 
Putting the estimates together in \eqref{eq:precis-final}, assuming $(\star)\leq 1$, 
\begin{align*} 
\dkol(W,G) \leq& \frac{2}{C I} \left( \frac{1}{(1-\alpha_0)^{3/2}} + \frac{24}{\pi \alpha_0} \right) (\star) + \sqrt{\frac{C-2}{C}} \frac{24}{\pi I \alpha_0} (\star) + \frac{4(C-2)}{(1-\alpha_0)^2 C I^2 \alpha_0} (\star) \\ 
\leq& \left[ \frac{2}{C} + \sqrt{\frac{C-2}{C}} \frac{24}{\pi I \alpha_0} + \frac{4(C-2)}{(1-\alpha_0)^2 C I^2 \alpha_0} \right] (\star) \ . 
\end{align*} 
The constant between brackets can be numerically estimated as being smaller than $1.06164$. 
 
If furthermore $c_k = \E[Y_k]$ for every $k\in V$, 
\begin{align} 
\rho = |\E[(S-\E[S])^3]| &\leq \sum_{i,j,k \in V} |\E[(Y_i-\E[Y_i])(Y_j-\E[Y_j])(Y_k-\E[Y_k])] | \label{eq:sum-rho-better}\\ 
&\leq \sum_{(i,j,k) \in \kV} \frac{1}{3}\left( \E[|Y_i-\E[Y_i]|^3] + \E[|Y_j-\E[Y_j]|^3] + \E[|Y_k-\E[Y_k]|^3] \right)\nonumber \\ 
&\leq \sum_{(i,j,k) \in \kV} \E[|Y_i-\E[Y_i]|^3]\nonumber 
\end{align} 
where we used the invariance of the sum by permutation of $(i,j,k)$ in the last line, and $\kV$ are the triplets $(i,j,k)$ that are connected in the dependency graph of $(Y_k)_{k\in V}$: indeed, if $(i,j,k)\notin \kV$ then one of $Y_i, Y_j, Y_k$, say $Y_i$ wlog, is independent from $(Y_j,Y_k)$ and the term in the sum \eqref{eq:sum-rho-better} is zero. 
For every $i\in V$ there are at most $3(D+1)^2$ couples $(j,k)$ with $(i,j,k)\in\kV$. It follows that $ \rho \leq 3(D+1)^2 \kA_3$. Adapting the computations yields a constant of $0.85771$ instead of $1.06164$.

\subsection{Proof with a finite third moment} 
\label{sec:proof-3+} 
 
We handle the case where the $(Y_k)_{k\in V}$ are not assumed to be almost surely uniformly bounded. 
In this section, we prove Theorem \ref{thm:dkol-delta-3+}, which contains our Berry--Esseen-type estimate in the case $\delta \in [3,\infty)$. Recall $W = (S-\E[S])/v$ and $\xi_\delta$ as in Definition \ref{def:renormedSD}, and let $G\sim\kN(0,1)$.

\begin{proposition}\label{prop:decrease-xi} 
	For all $1\le\delta\le\tilde\delta<\infty$, 
	\begin{equation*}\left(\frac{\kA_\delta}N\right)^{\frac 1\delta}\le \left(\frac{\kA_{\tilde\delta}}{N}\right)^{\frac{1}{\tilde\delta}}. 
	\end{equation*} 
	In particular, 
	\begin{equation*} 
		\xi_{\tilde\delta}\le\xi_\delta. 
	\end{equation*} 
	We also have $\xi_\delta\in[0,1]$ for every $\delta\in[2,\infty)$. 
\end{proposition} 
\begin{proof} 
	Write $(\Omega, \kF, \P)$ for the probability space on which the $(Y_k)_{k\in V}$ are defined. 
	Let $\mathsf Q= \frac1N\sum_{k\in V} \delta_{k}$ 
	be the uniform probability measure on $V$ and let $f:V\times\Omega \to[0,\infty) \, , \, (k, \omega) \mapsto \abs{Y_k(\omega)-c_k}^\delta$. Then 
	\begin{equation*} 
	\left(\frac{\kA_\delta}N\right)^{\frac 1\delta}\le \left(\frac{\kA_{\tilde\delta}}{N}\right)^{\frac{1}{\tilde\delta}} 
	\qquad \iff \qquad 
	\left(\int_{V\times\Omega} f\,\mathrm d\mathsf Q\otimes\P \right)^{\frac{\tilde\delta}{\delta}}  \le  \int_{V\times\Omega} f^{\frac{\tilde\delta}{\delta}}\,\mathrm d\mathsf Q\otimes\P, 
	\end{equation*} 
	which is a consequence of Jensen's inequality. 
	 
	It remains to prove that $\xi_\delta\in[0,1]$ for every $\delta\in[2,\infty)$. By the previously proven monotonicity of $\delta \mapsto \xi_\delta$, it is enough to prove $\xi_2\le 1$. For $i,j\in V$, write $i\sim j$ if $i$ is adjacent to $j$ in the dependency graph, or if $i=j$. Recall that $\Cov(Y_i,Y_j) = 0$ if we do not have $i\sim j$, and that by the Cauchy-Schwarz inequality and the inequality of arithmetic and geometric means, 
$$ \Cov(Y_i,Y_j) = \Cov(Y_i-c_i,Y_j-c_j)\leq \sqrt{\|Y_i-c_i\|_2\|Y_j-c_j\|_2} \leq \frac{1}{2} \left( \|Y_i-c_i\|_2^2 + \|Y_j-c_j\|_2^2 \right) \ . $$ 
Therefore 
	\begin{multline*} 
		v^2 = \V[S] = \sum_{i,j\in V} \Cov(Y_i, Y_j)  \le \frac{1}{2} \sum_{i\in V}\sum_{j\in V:j\sim i} \left( \|Y_i-c_i\|_2^2 + \|Y_j-c_j\|_2^2 \right)\\ \leq \kA_2 \sup_{i\in V} \sum_{j\in V: j\sim i} 1 \leq \kA_2 (D+1) \ . 
	\end{multline*} 
	It follows that $\xi_2=\sqrt{\frac{v^2}{\kA_2(D+1)}}\le 1$. 
\end{proof} 
 
\begin{lemma}\label{lem:fourier-delta-3+} 
Assume $\delta\geq 3$, as well as $v\neq 0$. Recall $C\in[5.17,5.18]$ as in Lemma \ref{lem:zone-of-control-bounded}. For every $s\in\R$ 
such that both of the following inequalities hold 
\begin{equation}\label{eq:cond-u-strong-3+} 
\begin{cases} 
|s| &\leq \frac{1}{6C} \sqrt{\frac{N}{D+1}} \xi_3^3 , \\ 
|s| &\leq \frac{1}{2\e}\sqrt{\frac{N}{D+1}} \left(\frac{1}{9}\xi_\delta^\delta\right)^{\frac{1}{\delta-2}}  , 
\end{cases} 
\end{equation} 
we have 
\begin{equation}\label{eq:Fourier-bound-d>=3} 
\left| \E\left[ \e^{\i s W} \right] - \e^{-\frac{s^2}{2}} \right| 
\ \leq 
\left(\frac{1}{\e} + \frac{3}{8 \e^2} \right) \left(\frac{2\e\abs s}{\xi_\delta}\right)^\delta \left(\frac{D+1}N\right)^{\frac{\delta-2}{2}} 
+ C  \left(\frac{|s|}{\xi_3}\right)^3 \sqrt{\frac{D+1}{N}}\exp\left( -\frac{s^2}{6} \right). 
\end{equation} 
\end{lemma}

\begin{proof} 
If $s=0$ there is nothing to prove, so assume $s\neq 0$. Recall $w$ given by \eqref{eq:def-u-J}. Then by Lemma \ref{lem:zone-of-control-unbounded}, for every $s\in\R$ and $0<L\leq v/(2\e|s|(D+1))$ that satisfies \eqref{eq:cond-u}, 
\begin{multline*} 
(D+1) \left| \E\left[ \e^{\i s W} \right] - \e^{-\frac{s^2}{2}} \right| 
\ \leq 
\left(2w + \frac{3 w^2}{2}\right) \frac{\kA_{\delta}}{L^{\delta}} \\ 
+ C \frac{\kA_3}{L^3} w^3 \ \exp\left( -\frac{w^2}{2(D+1)^2} \left(\frac{v^2}{L^2} - 3 (D+1) \frac{\kA_{\delta}}{L^{\delta}} - 2C(D+1) \frac{\kA_3}{L^3} w \right) \right) \ . 
\end{multline*} 
The infimum of the right-hand side is reached for $L = v/(2\e|s|(D+1))$ (provided that this $L$ satisfies \eqref{eq:cond-u}), giving $w=\frac{1}{2\e}$ and 
\begin{multline}\label{eq:long-calc} 
(D+1) \left| \E\left[ \e^{\i s W} \right] - \e^{-\frac{s^2}{2}} \right| 
\ \leq 
\left(\frac{1}{\e} + \frac{3}{8 \e^2} \right) \frac{\kA_\delta}{v^\delta} \left(2\e(D+1)|s|\right)^\delta \\ 
+ \frac{C}{8\e^3} \frac{\kA_3}{v^3} (2\e |s|(D+1))^3 \ \exp\Big(- \frac{|s|^2}{2} \Big( 1 - 3(D+1) \frac{\kA_\delta}{v^\delta} (2\e |s|(D+1))^{\delta-2} \\ 
- 2C(D+1) \frac{\kA_3}{v^3} (|s|(D+1)) \Big)\Big) \ . 
\end{multline} 
Then \eqref{eq:cond-u} is satisfied if and only if $3(D+1) \frac{\kA_{\delta}}{v^{\delta}} (2\e(D+1)|s|)^{\delta-2} < 1$. But by Definition of $\xi_\delta$, assumption \eqref{eq:cond-u-strong-3+} is equivalent to 
\begin{equation*} 
\begin{cases} 
2C (D+1)\frac{\kA_3}{v^3} (|s| (D+1)) &\leq \frac{1}{3} \ , \\ 
3(D+1) \frac{\kA_\delta}{v^\delta} (2\e |s| (D+1))^{\delta-2} &\leq \frac{1}{3} \ . 
\end{cases} 
\end{equation*} 
This implies that \eqref{eq:cond-u} holds, and using \eqref{eq:long-calc} we obtain furthermore that 
\begin{equation}\label{eq:Fourier-bound-d>=3-variant} 
	(D+1) \left| \E\left[ \e^{\i s W} \right] - \e^{-\frac{s^2}{2}} \right| \le \left(\frac{1}{\e} + \frac{3}{8 \e^2} \right) \frac{\kA_\delta}{v^\delta} \left(2\e(D+1)|s|\right)^\delta + C \frac{\kA_3}{v^3} (|s|(D+1))^3 \ \exp\left(-\frac{s^2}6\right). 
\end{equation} 
By Definition of $\xi_\delta$, the right-hand side of \eqref{eq:Fourier-bound-d>=3-variant} is equal to the right-hand side of \eqref{eq:Fourier-bound-d>=3}. 
\end{proof} 
 
\begin{proof}[Proof of Theorem \ref{thm:dkol-delta-3+}] 
We again follow the proof of Theorem \ref{thm:mod-phi-corollary-new}. Recall \eqref{eq:dkol-feller}: for every $t>0$, 
\begin{align*} 
\dkol\left(W,G\right) \leq& \frac{1}{\pi} \int_{-t}^t \left| \E\left[\e^{\i s W}\right] - \e^{-\frac{s^2}{2}} \right| \frac{1}{|s|} \d s + \frac{24}{t\pi\sqrt{2\pi}} . 
\end{align*} 
Provided that $t$ satisfies \eqref{eq:cond-u-strong-3+}, which we recall here: 
\begin{equation*} 
\begin{cases} 
t &\leq \frac{1}{6C} \sqrt{\frac{N}{D+1}} \xi_3^3 , \\ 
t &\leq \frac{1}{2\e}\sqrt{\frac{N}{D+1}} \left(\frac{\xi_\delta^\delta}{9}\right)^{\frac{1}{\delta-2}}  , 
\end{cases} 
\end{equation*} 
and which is implied by (recalling $\xi_\delta \leq \xi_3$ by Proposition \ref{prop:decrease-xi}) 
\begin{equation}\label{eq:cond-t-stronger} 
t \leq \frac{\xi_\delta^3}{18\e} \sqrt{\frac{N}{D+1}} \ , 
\end{equation} 
we can use Lemma \ref{lem:fourier-delta-3+}: 
\begin{align*} 
\dkol\left(W,G\right) 
&\leq \frac{1}{\pi} \int_{-t}^t \frac{1}{|s|} \left(\frac 1 \e + \frac{3}{8 \e^2} \right) \left(\frac{2\e |s|}{\xi_\delta}\right)^{\delta} \left(\frac{D+1}{N}\right)^{\frac{\delta-2}{2}} \d s \\ 
&+ \frac{1}{\pi} \int_{-t}^t  C \left(\frac{|s|}{\xi_3}\right)^{2} \sqrt{\frac{D+1}{N}} \ \exp\left( -\frac{s^2}{6} \right)  \frac{\d s}{\xi_3} 
+ \frac{24}{t\pi\sqrt{2\pi}} \ . 
\end{align*} 
Use the substitution $x=\frac{2\e s}{\xi_\delta} \left(\frac{D+1}N\right)^{\frac 12-\frac 1\delta}$ on the first term: 
\begin{equation*}\begin{split} 
\int_{-t}^t \frac{1}{|s|} \left(\frac 1 \e + \frac{3}{8 \e^2} \right) \left(\frac{2\e |s|}{\xi_\delta}\right)^{\delta} \left(\frac{D+1}{N}\right)^{\frac{\delta-2}{2}} \d s &= \left(\frac{2}{\e}+\frac{3}{4\e^2}\right)\int_{0}^t \frac{1}{|s|} \left(\frac{2\e |s|}{\xi_\delta}\right)^{\delta} \left(\frac{D+1}{N}\right)^{\delta\left(\frac{1}{2}-\frac{1}{\delta}\right)} \d s \\ 
&=\frac{8\e+3}{4\e^2} \int_0^{\frac{2\e t}{\xi_\delta} \left(\frac{D+1}{N}\right)^{\frac{1}{2}-\frac{1}{\delta}} } x^{\delta-1} \d x \\ 
&=\frac{8\e+3}{4\e^2\delta} \left(\frac{2\e t}{\xi_\delta} \left(\frac{D+1}{N}\right)^{\frac{1}{2}-\frac{1}{\delta}}\right)^{\delta} . 
\end{split}\end{equation*} 
The second integral simplifies to 
\begin{align*} 
\frac{C}{\pi \xi_3^3}\sqrt{\frac{D+1}{N}} \int_{-t}^t  s^{2} \ \exp\left( -\frac{s^2}{6} \right)  \d s 
\leq \frac{3 C \sqrt{6\pi}}{\pi \xi_3^3} \sqrt{\frac{D+1}{N}} . 
\end{align*} 
Therefore, 
\begin{align} 
\nonumber \dkol\left(W,G\right) 
&\leq \frac{8\e+3}{4\pi\e^2 \delta} \left( \frac{2\e t}{\xi_\delta} \left(\frac{D+1}{N}\right)^{\frac{1}{2}-\frac{1}{\delta}} \right)^{\delta} 
+ \frac{3 C \sqrt{6\pi}}{\pi \xi_3^3} \sqrt{\frac{D+1}{N}} 
+ \frac{24}{t\pi\sqrt{2\pi}} \\ 
&\leq \frac{8\e+3}{4\pi\e^2 \delta} \left( \frac{2\e t}{\xi_\delta} \left(\frac{D+1}{N}\right)^{\frac{1}{2}-\frac{1}{\delta}} \right)^{\delta} 
+ \left( \frac{C \sqrt{6\pi}}{6 \e \pi} + \frac{24}{\pi\sqrt{2\pi}}\right) \frac{1}{t}  \label{eq:dkol-intermediaire} 
\end{align} 
since $\frac{1}{\xi_3^3}\sqrt{\frac{D+1}{N}} \leq \frac{1}{18\e t}$ by \eqref{eq:cond-t-stronger} (using this bound lowers the sharpness of the final bound, but gives a simpler expression.)

Writing $(\star)(t)$ for the right-hand side of \eqref{eq:dkol-intermediaire}, seen as a function of $t$, we get 
\begin{align*} 
\dkol(W,G) &\leq \inf_{t \leq \frac{\xi_\delta^3}{18\e}\sqrt{\frac{N}{D+1}}} (\star) 
= \begin{cases} 
(\star)\left( \frac{\xi_\delta^3}{18\e}\sqrt{\frac{N}{D+1}} \right) &\text{ if } \frac{\xi_\delta^3}{18\e}\sqrt{\frac{N}{D+1}} < t_0 \\ 
 (\star)(t_0) &\text{ otherwise} 
 \end{cases} 
\end{align*} 
where $t_0$ is the minimum of $(\star)$ over $(0,\infty)$. $t_0$ exists and is unique, and is given by setting the derivative of $(\star)$ with respect to $t$ equal to $0$: 
\begin{align*} 
\left(\frac{1}{\xi_\delta}\left(\frac{D+1}{N}\right)^{\frac 1 2 - \frac 1 \delta}\right)^{\delta} (2\e t_0)^{\delta+1} = B \define 2\e \frac{4\pi\e^2}{8\e+3} \left( \frac{C \sqrt{6\pi}}{6 \e \pi} + \frac{24}{\pi\sqrt{2\pi}}\right) \in (71.107, 71.125)  \\ 
\iff  t_0 = \frac{1}{2\e} B^{\frac{1}{\delta+1}}  \left(\xi_\delta \left(\frac{N}{D+1}\right)^{\frac{1}{2}-\frac{1}{\delta}} \right)^{\frac{\delta}{\delta+1}} , 
\end{align*} 
which gives (noting that $\sup_{x\geq 3} (1+1/x) B^{-1/(x+1)} = 1$): 
\begin{align*} 
(\star)(t_0) 
&\leq 2\e \left(1+\frac{1}{\delta}\right) \left( \frac{C \sqrt{6\pi}}{6 \e \pi} + \frac{24}{\pi\sqrt{2\pi}}\right) B^{-\frac{1}{\delta+1}} \left(\xi_\delta \left(\frac{N}{D+1}\right)^{\frac{1}{2}-\frac{1}{\delta}}\right)^{-\frac{\delta}{\delta+1}} \\ 
&\leq 2\e \left( \frac{C \sqrt{6\pi}}{6 \e \pi} + \frac{24}{\pi\sqrt{2\pi}}\right) \left(\xi_\delta \left(\frac{N}{D+1}\right)^{\frac{1}{2}-\frac{1}{\delta}}\right)^{-\frac{\delta}{\delta+1}} \\ 
&\leq 18.96 \, \left(\xi_\delta \left(\frac{N}{D+1}\right)^{\frac{1}{2}-\frac{1}{\delta}}\right)^{-\frac{\delta}{\delta+1}} \ . 
\end{align*} 
Let us compare $(\star)(t)$ to 
$$ \varphi(t) \define \left(\frac{1}{\delta}+1\right) \left( \frac{C \sqrt{6\pi}}{6 \e \pi} + \frac{24}{\pi\sqrt{2\pi}}\right) \frac{1}{t} . $$ 
Using the formula for $t_0$, we can check that $\varphi(t) \leq (\star)(t_0)$ if and only if $t\geq t_0$, and when $t<t_0$, then 
$$ (\star)(t) \leq \varphi(t) , $$ 
so that 
$$ \dkol(W,G) \leq \max\left\{ \varphi\left( \frac{\xi_\delta^3}{18\e} \sqrt{\frac{N}{D+1}} \right) , (\star)(t_0) \right\} . $$ 
To finish the proof of the Theorem, it only remains to compute 
\begin{align*} 
\varphi\left( \frac{\xi_\delta^3}{18\e} \sqrt{\frac{N}{D+1}} \right) 
= \left(\frac{1}{\delta}+1\right) \frac{18\e}{\xi_\delta^3} \left( \frac{C \sqrt{6\pi}}{6 \e \pi} + \frac{24}{\pi\sqrt{2\pi}}\right) \sqrt{\frac{D+1}{N}} 
\leq \frac{227.5}{\xi_\delta^3} \sqrt{\frac{D+1}{N}} 
 \ . 
\end{align*} 
\end{proof}

\subsection{Proof with no finite third moment} \label{sec:proof-3-} 
 
In this section, we prove Theorem \ref{thm:result-refined-delta<2}, which considers the case $\delta \in (2,3]$. The method is close to that of Section \ref{sec:proof-3+}. Recall $W=(S-\E[S])/v$, $G\sim \kN(0,1)$, and $\xi_\delta$ as in Definition \ref{def:renormedSD}. 
 
\begin{lemma}\label{lem:fourier-delta-23} 
Assume $\delta\in (2,3)$, and assume $v>0$. For every $s$ such that 
\begin{equation}\label{eq:cond-u-strong-23} 
|s| \leq \left(\frac{\xi_\delta}{2\e}\right)^{\frac{\delta}{\delta-2}} \left(\frac{4\e^2}{3+C/\e}\right)^{\frac{1}{\delta-2}} \sqrt{\frac{N}{D+1}} \ , 
\end{equation} 
we have 
\begin{equation} 
\left| \E\left[ \e^{\i s W} \right] - \e^{-\frac{s^2}{2}} \right| 
\leq \left(\frac{N}{D+1}\right)^{1-\frac{\delta}{2}} \frac{C + 3\e + 8\e^2}{8\e^3} \left( \frac{2\e|s|}{\xi_\delta} \right)^\delta . 
\end{equation} 
\end{lemma}

\begin{proof} 
If $s=0$, then there is nothing to prove. Assume thus that $s\neq 0$. By Lemma \ref{lem:zone-of-control-unbounded}, for every $s$, for every $0 < L \leq v/(2\e |s| (D+1))$ that satisfies \eqref{eq:cond-u} (using $\delta'=\delta$), 
\begin{multline*} 
\frac{L^\delta}{\kA_\delta} (D+1) \left| \E\left[ \e^{\i s W} \right] - \e^{-\frac{s^2}{2}} \right| 
\ \leq 
2 w  + \frac{3}{2}w^2 \\ 
+ C w^3 \ \exp\left( -\frac{w^2}{2(D+1)^2} \left(\frac{v^2}{L^2} - 3 (D+1) \frac{\kA_{\delta}}{L^{\delta}} - 2C(D+1) \frac{\kA_\delta}{L^\delta} w \right)\right) . 
\end{multline*} 
Unlike for the case $\delta\geq 3$, the right-hand side may not be minimal when $L=v/(2\e |s| (D+1))$. Nevertheless, for the sake of simplicity we still take $L=v/(2\e |s| (D+1))$, giving $w=\frac1{2\e}$ and thus 
\begin{multline}\label{eq:almost-final-inequality-lem-delta-between-2-and-3} 
\frac{v^\delta}{(2\e \abs s (D+1))^\delta \kA_\delta} (D+1) \left| \E\left[ \e^{\i s W} \right] - \e^{-\frac{s^2}{2}} \right| 
\\ \leq 
\frac{1}{\e} + \frac{3}{8 \e^2} 
+ \frac{C}{8\e^3} \exp\left(-\frac1{8\e^2 (D+1)^2} \left((2\e\abs s (D+1))^2 - \frac{(2\e \abs s (D+1))^\delta \kA_\delta}{v^\delta} (D+1)\left(3+\frac C\e\right)\right)  \right) . 
\end{multline} 
Condition \eqref{eq:cond-u} becomes 
\begin{equation*} 
	(2\e\abs s(D+1))^{\delta-2}<\frac 1{3(D+1)}\frac{v^\delta}{\kA_\delta} \ . 
\end{equation*} 
But condition \eqref{eq:cond-u-strong-23} can be rewritten as 
\begin{equation*} 
(2\e \abs s (D+1))^{\delta-2} \leq \frac{1}{3 + C/\e}\xi_\delta^\delta (N(D+1))^{\frac{\delta-2}{\delta}} =\frac{1}{3+C/\e} \frac{1}{D+1}\frac{v^\delta}{\kA_\delta} \ . 
\end{equation*} 
In particular, \eqref{eq:cond-u} holds and 
\begin{equation*} 
	(2\e\abs s (D+1))^2 - \frac{(2\e \abs s (D+1))^\delta \kA_\delta}{v^\delta} (D+1)\left(3+\frac C\e\right)\ge0 \ , 
\end{equation*} 
so that from \eqref{eq:almost-final-inequality-lem-delta-between-2-and-3} 
\begin{equation*}\begin{split} 
	\abs{\E\left[ \e^{\i s W} \right]-\e^{-\frac{s^2}2}} &\le \frac{1}{D+1}\frac{(2\e\abs s (D+1))^\delta\kA_\delta}{v^\delta} \left(\frac 1\e+\frac 3{8\e^2}+\frac{C}{8\e^3}\right) \\&= \left(\frac 1\e+\frac 3{8\e^2}+\frac{C}{8\e^3}\right) \left(\frac{N}{D+1}\right)^{1-\frac{\delta}{2}} \left( \frac{2\e|s|}{\xi_\delta} \right)^\delta. 
\end{split}\end{equation*} 
\end{proof}

\begin{proof}[Proof of Theorem \ref{thm:result-refined-delta<2}] 
We follow the same method of proof as Theorem \ref{thm:dkol-delta-3+}. 
Recall \eqref{eq:dkol-feller}: 
\begin{align*} 
\dkol\left(W,G\right) \leq& \frac{1}{\pi} \int_{-t}^t \left| \E\left[\e^{\i s W}\right] - \e^{-\frac{s^2}{2}} \right| \frac{1}{|s|} \d s + \frac{24}{t\pi\sqrt{2\pi}} . 
\end{align*} 
Provided that $t$ satisfies \eqref{eq:cond-u-strong-23} (with $t$ replacing $|s|$), we can use Lemma \ref{lem:fourier-delta-23}: 
\begin{align} 
\nonumber \dkol\left(W,G\right) 
&\leq \frac{1}{\pi} \int_{-t}^t \frac{1}{|s|}  \left(\frac{N}{D+1}\right)^{1-\frac{\delta}{2}}  \frac{C + 3\e + 8\e^2}{8\e^3} \left( \frac{2\e|s|}{\xi_\delta} \right)^\delta \d s 
+ \frac{24}{t\pi\sqrt{2\pi}} \\ 
\label{eq:dkol-3--inter} &= \frac{C + 3\e + 8\e^2}{4\e^3 \pi\delta} \left(\frac{N}{D+1}\right)^{1-\frac{\delta}{2}} \left(\frac{2\e t}{\xi_\delta}\right)^\delta + \frac{24}{t\pi\sqrt{2\pi}} \ . 
\end{align} 
Write $(\star)$ for the right-hand side of \eqref{eq:dkol-3--inter}, seen as a function of $t$. It has a unique minimum $t_0$ over $(0,\infty)$, determined by cancelling its derivative: 
\begin{align*} 
\left(\frac{2\e}{\xi_\delta}\left(\frac{D+1}{N}\right)^{\frac 1 2 - \frac 1 \delta}\right)^{\delta} t_0^{\delta+1} = \frac{24}{\sqrt{2\pi}} \frac{4\e^3}{C + 3\e + 8\e^2} \\ 
\iff  t_0 = \left(\frac{\xi_\delta}{2\e}\right)^{\frac{\delta}{\delta+1}} \left( \frac{96\e^3}{(C + 3\e + 8\e^2)\sqrt{2\pi}} \right)^{\frac{1}{\delta+1}} \left(\frac{N}{D+1}\right)^{\frac{\delta-2}{2(\delta+1)}} \ , 
\end{align*} 
which gives, writing 
$$ B = \frac{96\e^3}{(C + 3\e + 8\e^2)\sqrt{2\pi}} \approx 10.62 \qquad , \qquad E = \frac{4\e^2}{3+C/\e} \ : $$ 
\begin{align} 
(\star)(t_0) 
&= \frac{\delta+1}{\delta}\frac{24}{\pi\sqrt{2\pi}} \left(\frac{\xi_\delta}{2\e}\right)^{-\frac{\delta}{\delta+1}} B^{-\frac{1}{\delta+1}} \left(\frac{D+1}{N}\right)^{\frac{\delta-2}{2(\delta+1)}} \ . \label{eq:star_0_23} 
\end{align} 
Let us define 
\begin{align}\label{eq:def-chi} 
\chi = \sup_{\delta\in(2,3)} \frac{\delta+1}{\delta}\frac{48\e}{\pi\sqrt{2\pi}} (2\e B)^{-\frac{1}{\delta+1}} \leq 8.015 \ . 
\end{align} 
Now write $t_1$ for the right-hand side of \eqref{eq:cond-u-strong-23}: we can check that 
\begin{align*} 
\dkol(W,G) \leq \inf \left\{ (\star)(t) \ : \ t\leq t_1 \right\} &= 
\begin{cases} 
 (\star)(t_1) &\text{ if } t_1 < t_0 \\ 
 (\star)(t_0) &\text{ otherwise} 
\end{cases} 
\end{align*} 
giving, by replacing $(\star)(t_0)$ by the upper bound obtained from \eqref{eq:star_0_23} and \eqref{eq:def-chi}, 
\begin{align}\label{eq:dkol-interm-disjonction-23} 
\dkol(W,G) \leq 
\begin{cases} 
 (\star)(t_1) &\text{ if } t_1 < t_0 \\ 
 \chi\left( \xi_\delta \left(\frac{N}{D+1}\right)^{\frac{1}{2}-\frac{1}{\delta}} \right)^{-\frac{\delta}{\delta+1}} &\text{ otherwise. } 
\end{cases} 
\end{align} 
We have $t_1 < t_0$ if and only if 
\begin{align} 
\nonumber &\left(\frac{\xi_\delta}{2\e}\right)^{\frac{\delta}{\delta+1}} B^{\frac{1}{\delta+1}} \left(\frac{N}{D+1}\right)^{\frac{\delta-2}{2(\delta+1)}}  >  E^{\frac{1}{\delta-2}} \left(\frac{\xi_\delta} {2\e}\right)^{\frac{\delta}{\delta-2}} \sqrt{\frac{N}{D+1}} \\ 
\nonumber \iff& \left(\frac{\xi_\delta}{2\e}\right)^{\delta-2} B^{\frac{\delta-2}{\delta}} \left(\frac{N}{D+1}\right)^{\frac{(\delta-2)^2}{2\delta}}  >   E^{\frac{\delta+1}{\delta}} \left(\frac{\xi_\delta} {2\e}\right)^{\delta+1} \left(\frac{N}{D+1}\right)^{\frac{(\delta-2)(\delta+1)}{2\delta}}\\ 
\nonumber \iff& \xi_\delta^3    <   B^{\frac{\delta-2}{\delta}} E^{-\frac{\delta+1}{\delta}} \left(\frac{N}{D+1}\right)^{\frac{(\delta-2)^2}{2\delta} - \frac{(\delta-2)(\delta+1)}{2\delta}} (2 \e)^3 \\ 
\nonumber \iff& \ \xi_\delta < 2\e\left[ B^{\delta-2} E^{-(\delta+1)} \right]^{\frac{1}{3\delta}} \cdot \left(\frac{N}{D+1}\right)^{\frac{1}{\delta}-\frac{1}{2}} , 
\end{align} 
in which case 
\begin{align*} 
\chi \left(\xi_\delta \left(\frac{N}{D+1}\right)^{\frac{1}{2}-\frac{1}{\delta}} \right)^{-\frac{\delta}{\delta+1}} &\geq \chi \left(2\e\left( B^{\delta-2} E^{-(\delta+1)} \right)^{\frac{1}{3\delta}}\right)^{\frac{-\delta}{\delta+1}} \geq 0.41973 \chi > 1 
\end{align*} 
provided that we choose $\chi \geq 2.3825$. 
This means that we can replace the right-hand side of \eqref{eq:dkol-interm-disjonction-23} by 
\begin{equation}\label{eq:final-rhs-23} 
\chi\left( \xi_\delta \left(\frac{N}{D+1}\right)^{\frac{1}{2}-\frac{1}{\delta}} \right)^{-\frac{\delta}{\delta+1}} \ . 
\end{equation} 
Indeed, if $t_1\geq t_0$ then \eqref{eq:final-rhs-23} is the right-hand side of \eqref{eq:dkol-interm-disjonction-23}, and if $t_1 < t_0$, then \eqref{eq:final-rhs-23} is larger than 1: since $\dkol(W,G)\leq 1$ we can replace the right-hand side of \eqref{eq:dkol-interm-disjonction-23} by any expression that is at least $1$, in particular by \eqref{eq:final-rhs-23}. 
This concludes the proof of the Theorem. 
\end{proof}

\printindex 
 
\defbibfilter{articles}{ 
	type=article or 
	type=misc 
} 
 
\section*{References} 
\subsection*{Articles} 
\printbibliography[filter=articles,heading=none] 
\subsection*{Books} 
\printbibliography[type=book,heading=none] 
 
\end{document}